\documentclass[10pt]{amsart}
\usepackage{amssymb, amsthm, amsmath}
\usepackage[all]{xy}
\usepackage{graphicx}
\usepackage{psfrag}

\newtheorem{prop}{Proposition}
\newtheorem{thm}{Theorem}
\newtheorem{cor}{Corollary}

\theoremstyle{definition}

\newtheorem{defn}{Definition}
\newtheorem{example}{Example}
\newtheorem{remark}{Remark}

\newcommand\A{{\mathbb A}}

\newcommand\N{{\mathbb N}}

\newcommand{\bR}{{\breve{\text{{\bf R}}}}}

\newcommand{\om}{{\varpi}}
\newcommand{\vp}{{\pi}}

\newcommand\Z{{\mathbb Z}}

\newcommand\bW{{\mathbb W}}

\newcommand\AG{{\mathfrak G}}
\newcommand\BG{{\mathfrak G}^B}
\newcommand\CG{{\mathfrak G}^C}
\newcommand\DG{{\mathfrak G}^D}

\newcommand{\be}{\beta}
\newcommand\la{\lambda}
\newcommand\s{{\sigma}}

\newcommand\ssm{\smallsetminus}

\newcommand\di{{\circ}}
\newcommand\eqto{\stackrel{\lower1.5pt\hbox{$\scriptstyle\sim\,$}}\to}
\newcommand\ov{\overline}

\newcommand\wt{\widetilde}

\DeclareMathOperator{\type}{\mathrm{type}}

\newcommand{\ignore}[1]{}

\begin{document}

\title[Tableau formulas for skew Grothendieck polynomials]
{Tableau formulas for skew Grothendieck polynomials}

\date{January 26, 2024}

\author{Harry~Tamvakis} \address{University of Maryland, Department of
Mathematics, William E. Kirwan Hall, 4176 Campus Drive, 
College Park, MD 20742, USA}
\email{harryt@umd.edu}



\begin{abstract}
An element of a Weyl group of classical type is skew if it is the left
factor in a reduced factorization of a Grassmannian element. The skew
Grothendieck polynomials are those which are indexed by skew elements
of the Weyl group. We define set-valued tableaux which are fillings of
the associated skew Young diagrams and use them to prove tableau
formulas for the skew double Grothendieck polynomials in all four
classical Lie types. We deduce tableau formulas for the Grassmannian
Grothendieck polynomials and the $K$-theoretic analogues of the
(double mixed) skew Stanley functions in the respective Lie types.
\end{abstract}

\maketitle

\setcounter{section}{-1}

\section{Introduction}

The double Grothendieck polynomials of Lascoux and Sch\"utzenberger
\cite{LS} and Kirillov and Naruse \cite{KN} represent the (stable)
Schubert classes in the equivariant $K$-theory of complete flag
manifolds, in each of the four classical Lie types. When the indexing
Weyl group element is skew, in the sense of \cite{T1, T3}, we
call these polynomials {\em skew Grothendieck polynomials}. The goal
of this article is to prove tableau formulas for the skew Grothendieck
polynomials, building on our earlier work \cite{T6}, which dealt with
the skew Schubert polynomials.

Each skew signed permutation is associated with a pair of (typed,
$k$-strict) partitions $\la\supset \mu$. We introduce {\em set-valued
  tableaux} on the skew Young diagram $\la/\mu$ by extending Buch's
definition \cite{B} in type A to types B, C, and D, in a way which is
natural from a Lie-theoretic point of view. Our combinatorial formulas
for Grothendieck polynomials are expressed as sums over set-valued
tableaux on this skew shape. The main results are the first such
theorems for symplectic and orthogonal Grothendieck polynomials, even
in the single case.

The skew elements of the symmetric group $S_n$ are the $321$-avoiding
or fully commutative permutations studied in \cite{BJS, St}. A
set-valued tableau formula for their double Grothendieck polynomials
was proved by Matsumura \cite{M2}, following earlier results in
\cite{ACT} and \cite{M1} for the single polynomials. As in \cite{T6},
the theorems of this paper are new even in type A, and provide an
alternative to the main result of \cite{M2}, which has a direct
analogue in types B, C, and D. We note that the (type A) skew stable
Grothendieck polynomials introduced and studied in \cite{FK1, B, LP,
Y} do not represent Schubert classes and therefore are different
than the ones found in \cite{ACT, M1, M2} and the present work.

The $K$-theoretic Stanley functions $F^B_w$, $F^C_w$, and $F^D_w$ lie
at the center of Kirillov and Naruse's approach to Grothendieck
polynomials. When $w$ is a skew element of the Weyl group, the theory
developed here also produces set-valued tableau formulas for
them. Since the lowest degree terms of these formal power series are
the ordinary skew Stanley functions, we deduce new combinatorial
formulas for the latter (Examples \ref{exCtab} and
\ref{exDtab}). These {\em barred $k$-tableau} and {\em barred typed
$k'$-tableau} formulas refine the $k$-tableau and typed $k'$-tableau
formulas of \cite{T1, T3} by separating the powers of $2$ which appear
in the multiplicities there.

Our theorems specialize to obtain formulas for the {\em Grassmannian
Grothendieck polynomials}, that is, the Grothendieck polynomials
indexed by Grassmannian elements. In types B and C, this answers a
question of Hudson, Ikeda, Matsumura, and Naruse
\cite[Sec.\ 1]{HIMN2}. The maximal Grassmannian (and fully
commutative) case of this problem was addressed earlier by Ikeda and
Naruse in \cite[Sec.\ 9]{IN} (see also \cite{GK}). Our formulas for
maximal Grassmannian polynomials differ from loc.\ cit.\ just as our
main result in type A differs from \cite[Thm.\ 3.1]{M2}.

The straightforward proofs were found by modifying the arguments of
\cite{T6}, this time employing the idCoxeter algebra, the Hecke
product on the Weyl group, and the new definitions of set-valued
tableaux in the orthogonal and symplectic Lie types. Other known
approaches to tableau formulas in type A do not suffice for our
purposes, since the Grassmannian elements in types B, C, and D are not
fully commutative. This is reflected, e.g., in the difference between
Proposition \ref{skewfactorA} and Propositions \ref{skewfactorC},
\ref{skewfactorD}, and their corollaries. Nevertheless, we arrive at a
general theory of such formulas which is uniform across the four
types, following \cite{T1, T3, T6}.

We remark that Schubert and Grothendieck polynomials do not give
{\em intrinsic formulas} for the Schubert classes, that is, formulas
which respect the symmetries of the underlying Weyl group
elements. Intrinsic formulas for the equivariant Schubert classes in
the equivariant cohomology ring of classical $G/P$ spaces and
corresponding ones in the theory of degeneracy loci of vector bundles
were obtained in \cite{T2} (in general) and \cite{T5} (for amenable
Weyl group elements). The problem of finding analogues of these
results in $K$-theory remains open. For progress on this question, we
refer the reader to \cite{BKTY}, which contains the general solution
in type A, and \cite{HIMN1}, which studies the Grassmannian loci in
types A, B, and C.

This article is organized as follows. Section \ref{prelims} contains
preliminary material on the relevant Weyl groups, Grothendieck
polynomials, $K$-theoretic Stanley functions, partitions, and
Grassmannian/skew (signed) permutations.  The following Sections
\ref{tAt}, \ref{tCt}, and \ref{tDt} deal with set-valued tableaux and
formulas for skew Grothendieck polynomials in the Lie types A, B/C,
and D, respectively.

\section{Preliminaries}
\label{prelims}

This section recalls some essential background definitions and
notation which will be used in this paper.  For more details on the
less standard among these, the reader may consult \cite{KN} and
\cite{T1, T3, T6}.

\subsection{Weyl groups and reduced words}

The Weyl group for the root system of type $\text{A}_{n-1}$ is the
{\em symmetric group} $S_n$ of permutations of the set
$\{1,\ldots,n\}$. The group $S_n$ is generated by the simple
transpositions $s_i=(i,i+1)$ for $1\leq i \leq n-1$.  The Weyl group
for the root system of type $\text{B}_n$ or $\text{C}_n$ is the {\em
hyperoctahedral group} $W_n$ of signed permutations on the set
$\{1,\ldots,n\}$. The group $W_n$ is generated by the transpositions $s_i$ for
$1\leq i \leq n-1$ and the sign change $s_0(1)=\ov{1}$ (here we set
$\ov{i}:=-i$ for any $i\geq 1$). The elements of $W_n$ are written in
one line notation as $n$-tuples $(w_1,\ldots, w_n)$, where $w_i:=w(i)$
for each $i\in [1,n]$.

There is a natural embedding of $W_n$ in $W_{n+1}$ defined by adding
the fixed point $n+1$, and we let $W_\infty :=\bigcup_n W_n$ and
$S_\infty:=\bigcup_n S_n$. The {\em length} of an element $w\in
W_\infty$, denoted $\ell(w)$, is the least integer $r$ such that we
have an equation $w=s_{a_1} \cdots s_{a_r}$. In this case, the word
$a_1\cdots a_r$ is called a {\em reduced word} for $w$.  We say that
$w$ is {\em decreasing down to $p$} if $w$ has a reduced word
$a_1\cdots a_r$ such that $a_1 > \cdots > a_r \geq p$, and that $w$ is
{\em increasing up from $p$} if $w$ has a reduced word $a_1\cdots a_r$
such that $p\leq a_1 < \cdots < a_r$. Here $p$ denotes a nonnegative
integer.

The Weyl group $\wt{W}_n$ for the root system of type $\text{D}_n$ is the
subgroup of $W_n$ consisting of all signed permutations with an even
number of negative values.  The group $\wt{W}_n$ is an extension of
$S_n$ by $s_\Box:=s_0s_1s_0$, an element which sends $(1,2)$ to
$(-2,-1)$ and fixes all integers $p\geq 3$. We define the natural
embedding $\wt{W}_n\hookrightarrow \wt{W}_{n+1}$ of Weyl groups as
above and set $\wt{W}_\infty:= \bigcup_n \wt{W}_n$.

The simple reflections in type D are indexed by the members of the set
$\N_\Box :=\{\Box,1,2,\ldots\}$, and the length and reduced words of
the elements of $\wt{W}_\infty$ are defined using them. Given $p\geq
1$, we define $w\in \wt{W}_n$ to be decreasing down to $p$ or
increasing up from $p$ in the same way as in types A, B, and
C. However, when $p=\Box$, we say that $w\in \wt{W}_n$ is {\em
  decreasing down to $\Box$} (resp.\ increasing up from $\Box$) if $w$
has a reduced word $a_1\cdots a_r$ which is a subword of
$(n-1,\ldots,2,\Box)$ (resp.\ a subword of $(\Box,2,\ldots,n-1)$).  In
all of the classical Lie types, if $w$ is decreasing down to $p$ or
increasing up from $p$, then the decreasing (resp.\ increasing) word
$a_1\cdots a_r$ for $w$ is uniquely determined.

An element of $W_n$ is called {\em unimodal} if it has a reduced word
$a_1\cdots a_r$ which is a subword of
$\Omega^B_n:=(n-1,\ldots,1,0,1,\ldots,n-1)$. For a unimodal $w\in
W_n$, the number of reduced words of $w$ which are subwords of
$\Omega^C_n:=(n-1,\ldots,1,0,0,1,\ldots,n-1)$ is equal to $2^{n(w)}$
for a nonnegative integer $n(w)$.  An element of $\wt{W}_n$ is called
{\em unimodal} if it has a reduced word $a_1\cdots a_r$ which is a
subword of $\Omega^D_n:=(n-1,\ldots,2,1,\Box,2,\ldots,n-1)$.  For a
unimodal $w\in \wt{W}_n$, the number of reduced words of $w$ which are
subwords of $\Omega^D_n$ is equal to $2^{n'(w)}$ for a nonnegative
integer $n'(w)$.

\subsection{Grothendieck polynomials and $K$-theoretic Stanley functions}
\label{wggp}

The double Grothendieck polynomials for the classical Lie groups
studied here are due to Lascoux and Sch\"utzenberger \cite{LS} and
Fomin and Kirillov \cite{FK1, FK2} (in type A) and Kirillov and Naruse
\cite{KN} (in types B, C, and D). Note that these objects depend on a
formal variable $\beta$, and in the latter three Lie types they are
not polynomials but power series of unbounded degree in the $Z$
variables. We refer the reader to \cite[Sec.\ 6]{KN} for the precise
way in which Grothendieck polynomials represent the stable equivariant
Schubert classes in the equivariant connective $K$-theory of complete
flag manifolds. When the parameter $\beta$ is set equal to zero, we
obtain the double Schubert polynomials defined by Ikeda, Mihalcea, and
Naruse \cite{IMN}, up to a change of sign in the $Y$ variables.

The idCoxeter algebra $\bW^{\be}_n$ of $W_n$ is the free unital
associative $\Z[\be]$-algebra generated by the elements
$\vp_0,\vp_1,\ldots,\vp_{n-1}$ modulo the relations
\[
\begin{array}{rclr}
\vp_i^2 & = & \be\vp_i & i\geq 0\, ; \\
\vp_i\vp_j & = & \vp_j\vp_i & |i-j|\geq 2\, ; \\
\vp_i\vp_{i+1}\vp_i & = & \vp_{i+1}\vp_i\vp_{i+1} & i>0\, ; \\
\vp_0\vp_1\vp_0\vp_1 & = & \vp_1\vp_0\vp_1\vp_0.
\end{array}
\]
For every $w\in W_n$, define $\vp_w := \vp_{a_1}\ldots \vp_{a_r}$, where
$a_1\cdots a_r$ is any reduced word for $w$. The elements $\vp_w$ for
$w\in W_n$ form a free $\Z[\be]$-basis of $\bW^{\be}_n$. We denote the
coefficient of $\vp_w\in \bW^{\be}_n$ in the expansion of the element
$\alpha\in \bW^{\be}_n$ by $\langle \alpha,w\rangle$.

Let $t$ be an indeterminate and define
\begin{gather*}
A_i(t) := (1+t \vp_{n-1})(1+t \vp_{n-2})\cdots (1+t \vp_i) \, ; \\ 
A'_i(t) := (1+t \vp_i)(1+t \vp_{i+1})\cdots (1+t \vp_{n-1}) \, ; \\
B(t) := (1+t \vp_{n-1})\cdots(1+t\vp_1)
(1+t\vp_0)(1+t \vp_1)\cdots (1+t \vp_{n-1}) \, ; \\
C(t) := (1+t \vp_{n-1})\cdots(1+t\vp_1)(1+t\vp_0)
(1+t\vp_0)(1+t \vp_1)\cdots (1+t \vp_{n-1}).
\end{gather*}
Suppose that $X=(x_1,x_2,\ldots)$, $Y=(y_1,y_2,\ldots)$, and
$Z=(z_1,z_2,\ldots)$ are three infinite sequences of commuting
independent variables. For any $\om\in S_n$, the type A Grothendieck
polynomial $\AG_\om$ is given by
\begin{equation}
\label{dbleA}
\AG_\om(X,Y) := \left\langle 
A'_{n-1}(y_{n-1})\cdots A'_1(y_1)A_1(x_1)\cdots A_{n-1}(x_{n-1}), \om\right\rangle,
\end{equation}
while its single version is defined by $\AG_\om(X):=\AG_\om(X,0)$.
If $A(X):=A_1(x_1)A_1(x_2)\cdots$ and $A'(Y):= \cdots A'_1(y_2)A'_1(y_1)$, then the
stable Grothendieck polynomial $G_\om$ of Fomin and Kirillov \cite{FK1} is given by
\[
G_\om(X,Y) := \lim_{m\to\infty}\AG_{1^m\times\om}(X,Y) =
\left\langle A'(Y)A(X), \om\right\rangle,
\]
where $1^m\times \om\in S_{m+n}$ is the permutation defined by $1^m\times \om(i)=i$ for
$i\in [1,m]$ and $1^m\times \om(i)=m+\om(i-m)$ if $i\in [m+1,m+n]$. Following 
Fomin and Kirillov (see \cite{FK1} and \cite[Cor.\ 6.5]{FK2}), the formal power series
$G_\om$ is a $K$-theoretic analogue of the (type A) double Stanley symmetric function.

Let $B(Z):=B(z_1)B(z_2)\cdots$, $C(Z):=C(z_1)C(z_2)\cdots$, and for
$w\in W_n$, define the type B and type C Grothendieck polynomials
$\BG_w$ and $\CG_w$ by
\begin{equation}
\label{dbleB}
\BG_w(Z;X,Y) := \left\langle 
A'_{n-1}(y_{n-1})\cdots A'_1(y_1)B(Z) A_1(x_1)\cdots 
A_{n-1}(x_{n-1}), w\right\rangle
\end{equation}
and
\begin{equation}
\label{dbleC}
\CG_w(Z;X,Y) := \left\langle 
A'_{n-1}(y_{n-1})\cdots A'_1(y_1)C(Z) A_1(x_1)\cdots 
A_{n-1}(x_{n-1}), w\right\rangle,
\end{equation}
while their single versions are given by $\BG_w(Z;X):=\BG_w(Z;X,0)$
and $\CG_w(Z;X):=\CG_w(Z;X,0)$, respectively. To compare with
\cite{KN}, note that the polynomial called ${\mathcal G}_w^B(a,b;x)$
in op.\ cit.\ would be the polynomial denoted by $\BG_w(x;a,b)$ here.
The polynomials $\BG_w$ and $\CG_w$ are stable under the inclusion of
$W_n$ in $W_{n+1}$; it follows that $\BG_w$, $\CG_w$, and $\AG_\om$ are
well defined for $w\in W_\infty$ and $\om\in S_\infty$, respectively.
The type B and type C $K$-theoretic Stanley functions $F^B_w$ and $F^C_w$
of \cite{KN} are given by $F^B_w(Z):=\left\langle B(Z),
w\right\rangle$ and $F^C_w(Z):=\left\langle C(Z), w\right\rangle$.

The idCoxeter algebra $\wt{\bW}^{\be}_n$ of the group $\wt{W}_n$ is the free unital
associative $\Z[\be]$-algebra generated by the elements
$\vp_\Box,\vp_1,\ldots,\vp_{n-1}$ modulo the relations
\[
\begin{array}{rclr}
\vp_i^2 & = & \be\vp_i & i\in \N_\Box\, ; \\
\vp_\Box \vp_1 & = & \vp_1 \vp_\Box\, ; \\
\vp_\Box \vp_2 \vp_\Box & = & \vp_2 \vp_\Box \vp_2\, ; \\
\vp_i\vp_{i+1}\vp_i & = & \vp_{i+1}\vp_i\vp_{i+1} & i>0\, ; \\
\vp_i\vp_j & = & \vp_j\vp_i & j> i+1, \ \text{and} \ (i,j) \neq (\Box,2).
\end{array}
\]

For any element $w\in \wt{W}_n$, choose a reduced word $a_1\cdots a_r$
for $w$, and define $\vp_w := \vp_{a_1}\ldots \vp_{a_r}$.  Denote the
coefficient of $\vp_w\in \wt{\bW}^{\be}_n$ in the expansion of the
element $\alpha\in \wt{\bW}^{\be}_n$ in the $\vp_w$ basis by $\langle
\alpha,w\rangle$, and define
\[
D(t) := (1+t \vp_{n-1})\cdots (1+t \vp_2)(1+t \vp_1)(1+t \vp_\Box)
(1+t \vp_2)\cdots (1+t \vp_{n-1}).
\]
Let $D(Z):=D(z_1)D(z_2)\cdots$, and for $w\in \wt{W}_n$, define
the type D Grothendieck polynomial $\DG_w$ by 
\begin{equation}
\label{dbleD}
\DG_w(Z;X,Y) := \left\langle 
A'_{n-1}(y_{n-1})\cdots A'_1(y_1) D(Z) A_1(x_1)\cdots 
A_{n-1}(x_{n-1}), w\right\rangle,
\end{equation}
and set $\DG_w(Z;X):=\DG_w(Z;X,0)$. The Grothendieck polynomial
$\DG_w(Z;X,Y)$ is stable under the natural inclusions
$\wt{W}_n\hookrightarrow \wt{W}_{n+1}$, and hence is well defined for
$w\in \wt{W}_\infty$.  Following \cite{KN}, the type D $K$-theoretic
Stanley function $F^D_w$ is defined by $F^D_w(Z):=\left\langle D(Z),
w\right\rangle$.

Given any Weyl group elements $u,v,w$, we say that $w$ is the {\em
  Hecke product} $u\di v$ of $u$ and $v$ if $\vp_u\vp_v =
\be^{\ell(u)+\ell(v)-\ell(w)}\vp_w$ in the corresponding idCoxeter
algebra. The Hecke product $\di$ is characterized by the relations
$s_i\di w = s_iw$, if $\ell(s_iw)>\ell(w)$, and $s_i\di w = w$, if
$\ell(s_iw)<\ell(w)$, for each simple reflection $s_i$.  This defines
an associative product on the Weyl group such that $u\di v = uv$ if
and only if $\ell(uv)=\ell(u)+\ell(v)$. In the latter case, we say
that the product $uv$ is {\em reduced}.  If $u_1\di\cdots\di u_r =w$,
then $u_1\di\cdots\di u_r$ is called a {\em Hecke factorization} of
$w$.

\subsection{Partitions and Grassmannian/skew Weyl group elements}
\label{gpps}

\subsubsection{Type A}
\label{typeA}

A {\em partition} $\la=(\la_1,\la_2,\ldots,\la_r)$ is a finite weakly decreasing
sequence of nonnegative integers, which we identify with its Young diagram of
boxes. The {\em length} of $\la$ is the number of non-zero parts $\la_i$.
The containment relation $\mu\subset\la$ of partitions is defined by
using their diagrams, and the set-theoretic difference $\la\ssm\mu$ is
the {\em skew diagram} $\la/\mu$. A skew diagram is a {\em horizontal strip}
(resp.\ {\em vertical strip}) if it does not contain two boxes in the
same column (resp.\ row), and, following \cite[Sec.\ 4]{B}, a {\em
  rook strip} if it contains no two boxes in the same row or column.

Fix a nonnegative integer $m$. An element $\om$ in $S_\infty$ is {\em
  $m$-Grassmannian} if $\ell(\om s_i)>\ell(\om)$ for all $i\neq m$.
Every $m$-Grassmannian permutation $\om$ corresponds to a unique
partition $\la$ of length at most $m$, called the {\em shape} of
$\om$, satisfying $\la=(\om_m-m,\ldots,\om_1-1)$.  When the shape
$\la$ and $m$ are given, we denote the corresponding permutation by
$\om_\la$.

A permutation $\om\in S_\infty$ is called {\em skew} if there exists
an $m$-Grassmannian permutation $\om_\la$ (for some $m$) and a reduced
factorization $\om_\la = \om\om'$ in $S_\infty$. In this case, the
right factor $\om'$ equals $\om_\mu$ for some $m$-Grassmannian
permutation $\om_\mu$, and we have $\mu\subset\la$.  We say that $\om$
is associated to the pair $(\la,\mu)$ and write $\om=\om_{\la/\mu}$.
Moreover, we have $\ell(\om)=|\la/\mu|:=\sum_i (\la_i-\mu_i)$, and
there is a 1-1 correspondence between reduced factorizations $uv$ of
$\om$ and partitions $\nu$ with $\mu\subset\nu\subset\la$.

\subsubsection{Types B and C}
\label{typesBC}

Fix a nonnegative integer $k$. We denote the box in row $r$ and column
$c$ of a Young diagram by $[r,c]$, and call the box $[r,c]$ a {\em
  left box} if $c \leq k$ and a {\em right box} if $c>k$. We say that
the boxes $[r,c]$ and $[r',c']$ are {\em $k$-related} if $|c-k-1|+r =
|c'-k-1|+r'$ and {\em $k'$-related} if $|c-k-\frac{1}{2}|+r =
|c'-k-\frac{1}{2}|+r'$.

A partition $\la$ is called {\em $k$-strict} if no part greater than
$k$ is repeated. The number of parts $\la_i$ of $\la$ which are
greater than $k$ is the {\em $k$-length} of $\la$, denoted by
$\ell_k(\la)$. We define $\la_0:=\infty$ and agree that the diagram of
$\la$ includes all boxes $[0,c]$ in row zero.  The {\em rim} of $\la$
is the set of boxes $[r,c]$ of its Young diagram such that box
$[r+1,c+1]$ lies outside of the diagram.

If $\mu\subset\la$ are two $k$-strict partitions, we let $R$
(resp.\ $\A$) denote the set of right boxes of $\mu$ (including boxes
in row zero) which are bottom boxes of $\la$ in their column and are
(resp.\ are not) $k'$-related to a left box of $\la/\mu$. The pair
$\mu\subset\la$ forms a {\em $k$-horizontal strip} $\la/\mu$ if (a)
$\la/\mu$ is contained in the rim of $\la$, and the right boxes of
$\la/\mu$ form a horizontal strip; (b) no two boxes in $R$ are
$k'$-related; and (c) if two boxes of $\la/\mu$ lie in the same
column, then they are $k'$-related to exactly two boxes of $R$, which
both lie in the same row.  We let $n(\la/\mu)$ denote the number of
connected components of $\A$ which do not have a box in column
$k+1$. Here two boxes in $\A$ are connected if they share a vertex or
an edge.

An element $w$ of $W_\infty$ is {\em $k$-Grassmannian} if it satisfies
$\ell(ws_i)>\ell(w)$ for all $i\neq k$.  Each $k$-Grassmannian element
$w \in W_\infty$ corresponds to a unique $k$-strict partition $\la$,
called the {\em shape} of $w$.  We have
\[
\la_i=\begin{cases} 
k+|w_{k+i}| & \text{if $w_{k+i}<0$}, \\
\#\{j\in [1,k]\, :\, w_j> w_{k+i}\} & \text{if $w_{k+i}>0$}.
\end{cases}
\]
A formula for the inverse of this map is found in
\cite[Sec.\ 3.1]{T6}.  Note that if $w$ lies in $S_\infty$, then its
shape in types B and C is the partition whose diagram is the transpose
of its shape in type A.  If the shape $\la$ and $k$ are given, then we
denote the corresponding Weyl group element by $w_\la$.

\subsubsection{Type D}

Fix a positive integer $k$.  A {\em typed $k$-strict partition} is a
pair consisting of a $k$-strict partition $\la$ together with an
integer $\type(\la)\in \{0,1,2\}$, which is positive if and only if
$\la_i=k$ for some index $i$.

If $\mu\subset\la$ are two typed $k$-strict partitions, the sets $R$
and $\A$ are defined as in types B and C, replacing `$k'$-related' by
`$(k-1)$-related'. A pair $\mu\subset\la$ of typed $k$-strict
partitions forms a {\em typed $k'$-horizontal strip} $\la/\mu$ if
$\type(\la)+\type(\mu)\neq 3$ and conditions (a), (b), and (c) in the
definition of a $k$-horizontal strip hold, again replacing
`$k'$-related' by `$(k-1)$-related'.  We define $n'(\la/\mu)$ to be
one less than the number of connected components of $\A$.

An element $w\in \wt{W}_\infty$ has type $0$ if $|w_1|=1$, type $1$ if
$w_1>1$, and type 2 if $w_1<-1$. We say that $w\in \wt{W}_\infty$ is
{\em $k$-Grassmannian} if $\ell(ws_i)>\ell(w)$ for all $i\neq k$, if
$k>1$, and for all $i\notin \{\Box,1\}$, if $k=1$.  There is a type-preserving
bijection between the $k$-Grassmannian elements $w$ of $\wt{W}_\infty$
and typed $k$-strict partitions $\la$. We have
\[
\la_i=\begin{cases} 
k-1+|w_{k+i}| & \text{if $w_{k+i}<0$}, \\
\#\{j\in [1,k]\, :\, |w_j|> w_{k+i}\} & \text{if $w_{k+i}>0$}.
\end{cases}
\]
See \cite[Sec.\ 4.1]{T6} for a description of inverse of this map.
If the Weyl group element $w$ corresponds to the typed $k$-strict
partition $\la$, then we denote $w=w(\la,k)$ by $w_\la$.

\medskip

A Weyl group element $w$ in Lie type B or C (resp.\ D) is called {\em
  skew} if there exists a $k$-Grassmannian element $w_\la$ for some
$k\geq 0$ (resp.\ for some $k\geq 1$) and a reduced factorization
$w_\la = ww'$ in $W_\infty$ (resp.\ $\wt{W}_\infty$).  In this case,
the right factor $w'$ equals $w_\mu$ for some $k$-Grassmannian element
$w_\mu$, and we have $\mu\subset\la$. We say that $(\la,\mu)$ is a
{\em compatible pair} of (typed) $k$-strict partitions and that $w$ is
associated to the pair $(\la,\mu)$. We write $w=w_{\la/\mu}$ and note
that $\ell(w)=|\la/\mu|$. For example, any $k$-horizontal strip
(resp.\ typed $k'$-horizontal strip) $\la/\mu$ is a compatible pair
$(\la,\mu)$ of (typed) $k$-strict partitions.

We say that a sequence of (typed) $k$-strict partitions $ \la^0\subset
\la^1 \subset \cdots \subset \la^p $ is {\em compatible} if
$(\la^i,\la^{i-1})$ is a compatible pair for each $i\in [1,p]$. If
$p=2$ we say that $\la^0\subset \la^1 \subset\la^2$ is a {\em
compatible triple} of such partitions.  There is a 1-1
correspondence between reduced factorizations $uv$ of $w_{\la/\mu}$
and (typed) $k$-strict partitions $\nu$ such that
$\mu\subset\nu\subset\la$ is a compatible triple, with
$u=w_{\la/\nu}$ and $v=w_{\nu/\mu}$.

\section{Tableau formula for type A skew Grothendieck polynomials}
\label{tAt}

In this section, we assume that all Grassmannian and skew permutations are
taken with respect to a fixed positive integer $m$.
An $m$-Grassmannian element $\om\in S_n$ satisfies
$\ell(s_i\om)<\ell(\om)$ if and only if $\om=(\cdots i+1 \cdots | \cdots
i \cdots)$, where the vertical line $|$ lies between $\om_m$ and
$\om_{m+1}$. Using this, it is easy to see that a skew permutation
$\om_{\la/\mu}$ is decreasing down to $1$ (resp.\ increasing
up from $1$) if and only if $\la/\mu$ is a horizontal
(resp.\ vertical) strip.

\begin{prop}
\label{skewfactorA}
Suppose that $\om_{\la/\mu}=u\di v$ for some skew permutation
$\om_{\la/\mu}$. Then there exist partitions $\nu$ and $\rho$ with
$\mu\subset\rho\subset\nu\subset\la$ such that $\nu/\rho$ is a rook
strip, $u=\om_{\la/\rho}$, and $v=\om_{\nu/\mu}$. Moreover, the
converse statement also holds.
\end{prop}
\begin{proof}
Suppose that $u=s_{a_1}\cdots s_{a_r}$ with $r=\ell(u)$, so that
$\om_{\la/\mu} = s_{a_1}\di\cdots \di s_{a_r} \di v$. Since $s_i\di v$
is equal to $v$ or to the reduced product $s_iv$ for any $i$, it is
clear that there exists a reduced factorization $\om_{\la/\mu}= \om'v$
for some $\om'$, and hence that $v=\om_{\nu/\mu}$ for some partition
$\nu$ with $\mu\subset\nu\subset\la$.

The assertions regarding $u$ and $\rho$ are proved using induction on
$\ell(u)$. If $\ell(u)=1$ then $u=s_i$ for some $i$, so
$\om_{\la/\mu}=s_i\di \om_{\nu/\mu}$. If
$s_i\di\om_{\nu/\mu}=s_i\om_{\nu/\mu}$, then we set $\rho:=\nu$. If
$s_i\di\om_{\nu/\mu}=\om_{\nu/\mu}$, then $\nu=\la$ and
$s_i\om_{\nu/\mu}=\om_{\rho/\mu}$ for some partition $\rho$ with
$\mu\subset\rho\subset\nu$ and $|\rho|=|\nu|-1$, which satisfies the
required property.

Assume next that $u=s_iu'$ with $\ell(u)=\ell(u')+1$ and
$\om_{\la/\mu}=u\di\om_{\nu/\mu} = s_i\di u' \di \om_{\nu/\mu}$. It
follows as above that $u'\di \om_{\nu/\mu}=\om_{\zeta/\mu}$ for some
partition $\zeta\subset\la$, so $s_i\di\om_\zeta=\om_\la$.  By the
inductive hypothesis, we have $u'=\om_{\zeta/\rho}$ for some partition
$\rho\subset\nu$ such that $\nu/\rho$ is a rook strip. If
$|\zeta|=|\la|-1$, then $s_i\om_\zeta=\om_\la$ is a reduced
factorization, hence $u=\om_{\la/\rho}$ and we are done.  If
$\zeta=\la$, then $\ell(s_i\om_\la)=\ell(\om_\la)-1$ implies that
$\om_\la=(\cdots i+1 \cdots | \cdots i \cdots )$, that is,
$\om_\la(e)=i+1$ and $\om_\la(f)=i$ where $e\leq m < f=m+1+i-e$, and
$s_i\om_\la=\om_{\la\ssm b}$ for some box $b$ in row $m+1-e$. Now
$\ell(s_i\om_{\la/\rho})=\ell(s_iu')=\ell(u')+1$ implies that $b$ must
lie in $\rho$. (Indeed, if not then we have $\rho\subset (\la\ssm b)$,
hence a reduced factorization $w_{\la\ssm b} = vw_\rho$ for some $v$,
therefore $u'=s_iv$ is also reduced and so $\ell(s_iu')=\ell(u')-1$.)
Since both $\la\ssm b$ and $\wt{\rho}:=\rho\ssm b$ are partitions, we
must also have $\om_\rho(e) = i+1$ and $\om_\rho(f)=i$.

If $s_{a_1}\cdots s_{a_r}$ is a reduced word of $u'=\om_{\la/\rho}$,
then we deduce that $\{a_1,\ldots, a_r\}\cap \{i-1,i,i+1\}=\emptyset$,
and hence that $s_i\om_{\la/\rho}=\om_{\la/\rho}s_i$. Now
$s_i\om_{\wt{\rho}}=\om_\rho$ and therefore
$u\om_{\wt{\rho}}=s_i\om_{\la/\rho}\om_{\wt{\rho}}=\om_{\la/\rho}\om_{\rho}=\om_{\la}$.
We conclude that $u=\om_{\la/\wt{\rho}}$. Finally, since $\la\ssm b$
is a partition and $\rho\subset\nu\subset\la$, it follows that $b$ is
the only box in its row and column in $\nu/\wt{\rho}$, and therefore
that $\nu/\wt{\rho}=(\nu/\rho)\cup b$ is a rook strip.

To prove the converse, suppose that $\mu\subset\rho\subset\nu\subset\la$ are
such that $\nu/\rho$ is a rook strip. Then it is easy to see by induction on
$|\nu/\rho|$ that $\om_{\nu/\rho}\circ \om_{\nu/\mu} = \om_{\nu/\mu}$. Therefore, we have
$\om_{\la/\rho}\circ \om_{\nu/\mu} = \om_{\la/\nu}\circ\om_{\nu/\rho}\circ \om_{\nu/\mu}=
\om_{\la/\nu}\circ\om_{\nu/\mu}=\om_{\la/\mu}$.
\end{proof}

\begin{cor}
\label{Afact}
For the skew element $\om_{\la/\mu}\in S_\infty$, we have
\[
\AG_{\om_{\la/\mu}}(X,Y) = \sum_{\rho,\nu} \AG_{\om_{\nu/\mu}}(X)\AG_{\om_{\la/\rho}^{-1}}(Y)
\]
summed over all partitions $\rho, \nu$ with $\mu\subset\rho\subset\nu\subset\la$
such that $\nu/\rho$ is a rook strip.
\end{cor}
\begin{proof}
The definition (\ref{dbleA}) of $\AG_\om$ implies that
\[
\AG_\om(X,Y) = \sum_{u\di v=\om} \AG_v(X)\AG_{u^{-1}}(Y)
\]
summed over all Hecke factorizations $u\di v$ of $\om$. The result follows immediately from
this and Proposition \ref{skewfactorA}.
\end{proof}

Let $\la$ and $\mu$ be any two partitions of length at most $m$ with
$\mu\subset\la$, and choose $n\geq 1$ such that $\om_\la\in S_n$. Let {\bf P}
denote the ordered alphabet 
\[
(n-1)'<\cdots<1'<1<\cdots<n-1. 
\] 
The symbols $(n-1)',\ldots,1'$ are said to be {\em marked}, while the rest
are {\em unmarked}. Given a subset $A$ of {\bf P}, an
any element of $A$ will be called an {\em entry}.
If $A$ and $B$ are two nonempty subsets of an ordered set of symbols, we
write $A\leq B$ if $\max(A) \leq \min(B)$.

\begin{defn}
\label{AGtdef}
A {\em set-valued $m$-bitableau} $U$ of shape $\la/\mu$ is a filling
of the boxes in $\la/\mu$ with finite nonempty subsets of {\bf P}
which is weakly increasing along each row and down each column, such
that (i) the marked (resp.\ unmarked) entries are strictly increasing
each down each column (resp.\ along each row), and (ii) the entries in
row $i$ lie in the interval $[(\mu_i+m+1-i)',\la_i+m-i]$ for each
$i\in [1,m]$.  We let $|U|$ be the total number of symbols of {\bf P}
which appear in $U$, including their multiplicities, and define
\[
(xy)^U:= \prod_ix_i^{n'_i}\prod_i y_i^{n_i} 
\]
where $n'_i$ (resp.\ $n_i$) denotes the number of times that $i'$
(resp.\ $i$) appears in $U$.
\end{defn}

\begin{thm}
\label{Askew} For the skew permutation $\om:=\om_{\la/\mu}$ in $S_n$, we have
\begin{equation}
\label{AGteq}
\AG_\om(X,Y)=\sum_U \be^{|U|-|\la/\mu|}(xy)^U
\end{equation}
summed over all set-valued $m$-bitableaux $U$ of shape $\la/\mu$.
\end{thm}
\begin{proof}
It follows from formula (\ref{dbleA}) that we have 
\[
\AG_\om(X,Y) = \sum \be^{\ell(u,v,\om)}
y_{n-1}^{\ell(v_{n-1})}\cdots y_1^{\ell(v_1)} x_1^{\ell(u_1)}\cdots x_{n-1}^{\ell(u_{n-1})}
\]
where the sum is over all Hecke factorizations $v_{n-1}\di\cdots \di
v_1\di u_1\di\cdots \di u_{n-1}$ of $\om$ such that $v_p$ is
increasing up from $p$ and $u_p$ is decreasing down to $p$ for each
$p\in [1,n-1]$, while $\ell(u,v,\om):=\sum_i\ell(u_i)+\sum_i\ell(v_i)
- \ell(\om)$.
Proposition \ref{skewfactorA} shows that such
factorizations correspond to pairs of sequences of partitions
\[
\mu = \la^0 \subset \la^1 \subset \cdots \subset \la^{2n-2} =\la
\quad \mathrm{and} \quad
\mu = \nu^0, \nu^1, \ldots, \nu^{2n-2} = \la
\]
such that (i) $\mu\subset\nu^i\subset\la^i$ and $\la^i/\nu^i$ is a rook strip
for each $i$, and (ii) $\la^i/\nu^{i-1}$ is a horizontal strip for
$1\leq i\leq n-1$ and a vertical strip for $n\leq i\leq 2n-2$. We have
$u_{n-i}=\om_{\la^i/\nu^{i-1}}$, for $i\in [1,n-1]$, and
$v_{i+1-n}=\om_{\la^i/\nu^{i-1}}$, for $i\in [n,2n-2]$. Note that some factors
in the product $v_{n-1}\di\cdots \di v_1\di u_1\di\cdots \di u_{n-1}$ may be
trivial, and in this case the associated skew diagram $\la^i/\nu^{i-1}$ is
empty. We obtain a corresponding filling $U$ of the boxes in $\la/\mu$ by
including the entry $(n-i)'$ in each box of $\la^i/\nu^{i-1}$ for
$1\leq i\leq n-1$ and the entry $i+1-n$ in each box of
$\la^i/\nu^{i-1}$ for $n\leq i\leq 2n-2$. The required bounds on these
entries are established exactly as in the proof of \cite[Thm.\ 1]{T6},
showing that $U$ is a set-valued $m$-bitableau of shape $\la/\mu$ such
that $(xy)^U=y_{n-1}^{\ell(v_{n-1})}\cdots
y_1^{\ell(v_1)}x_1^{\ell(u_1)}\cdots x_{n-1}^{\ell(u_{n-1})}$ and
$|U|=\sum_i\ell(u_i)+\sum_i\ell(v_i)$. Since the sum in equation
(\ref{AGteq}) is over all such $U$, the result follows.
\end{proof}

\begin{example}
For each positive integer $m$, the permutation $s_m$ of length one can be
viewed as the $m$-Grassmannian element $\om_\la$ associated to the partition
$\la=1$. Theorem \ref{Askew} therefore gives 
\[
\AG_{s_m}(X,Y) = \sum_E \beta^{|E|-1}x_1^{e_1}\cdots x_m^{e_m}y_1^{e_{m+1}}
\cdots y_m^{e_{2m}}
\]
summed over all $4^m-1$ vectors $E=(e_1,\ldots,e_{2m})$ with $e_i\in \{0,1\}$
for each $i$ and $|E|:=e_1+\cdots + e_{2m}>0$. This corrects an error in
\cite[Example 1]{BKTY}.
\end{example}

\begin{example}
\label{skstabgr}
Let $\om_{\la/\mu}$ be a skew element of $S_\infty$. Extend the
alphabet {\bf P} to include all marked and unmarked integers, and omit
the bounds on the entries of the $m$-bitableaux found in Definition
\ref{AGtdef}.  Then the right hand side of equation (\ref{AGteq})
gives a tableau formula for the {\em skew stable Grothendieck polynomial}
$G_{\la/\mu}(X,Y)$ of Fomin and Kirillov \cite{FK1}. This result was
indicated by Buch in \cite[Remark 3.3]{B}.
\end{example}

\section{Tableau formula for type C skew Grothendieck polynomials}
\label{tCt}

In this section, we assume that all Grassmannian
and skew elements are taken with respect to a fixed nonnegative integer $k$.

Let $\la$ and $\mu$ be any two $k$-strict partitions such that
$(\la,\mu)$ is a compatible pair and let $w_{\la/\mu}=w_\la
w_\mu^{-1}$ be the corresponding skew element of $W_\infty$. According
to \cite[Prop.\ 5]{T1}, $w_{\la/\mu}$ is unimodal if and only if
$\la/\mu$ is a $k$-horizontal strip, and in this case we have
$n(w_{\la/\mu})=n(\la/\mu)$. If $w_{\la/\mu}$ is decreasing down to
$0$ (resp.\ increasing up from $0$), then we say that the
$k$-horizontal strip $\la/\mu$ is a {\em $z$-strip} (resp.\ {\em
  $\ov{z}$-strip}). If we furthermore have $w_{\la/\mu}\in S_\infty$,
we also say that $\la/\mu$ is an {\em $x$-strip} (resp.\ {\em
  $y$-strip}). A $z$-strip (resp.\ $\ov{z}$-strip) is an $x$-strip
(resp.\ $y$-strip) if and only if $\ell_k(\la)=\ell_k(\mu)$.

\begin{example}
The shaded boxes in the Young diagrams below
illustrate two $k$-horizontal strips with $k=6$. The one
on the left is a $z$-strip and the one on the right is both
a $\ov{z}$-strip and a $y$-strip.
\medskip
\[
\includegraphics[scale=0.32]{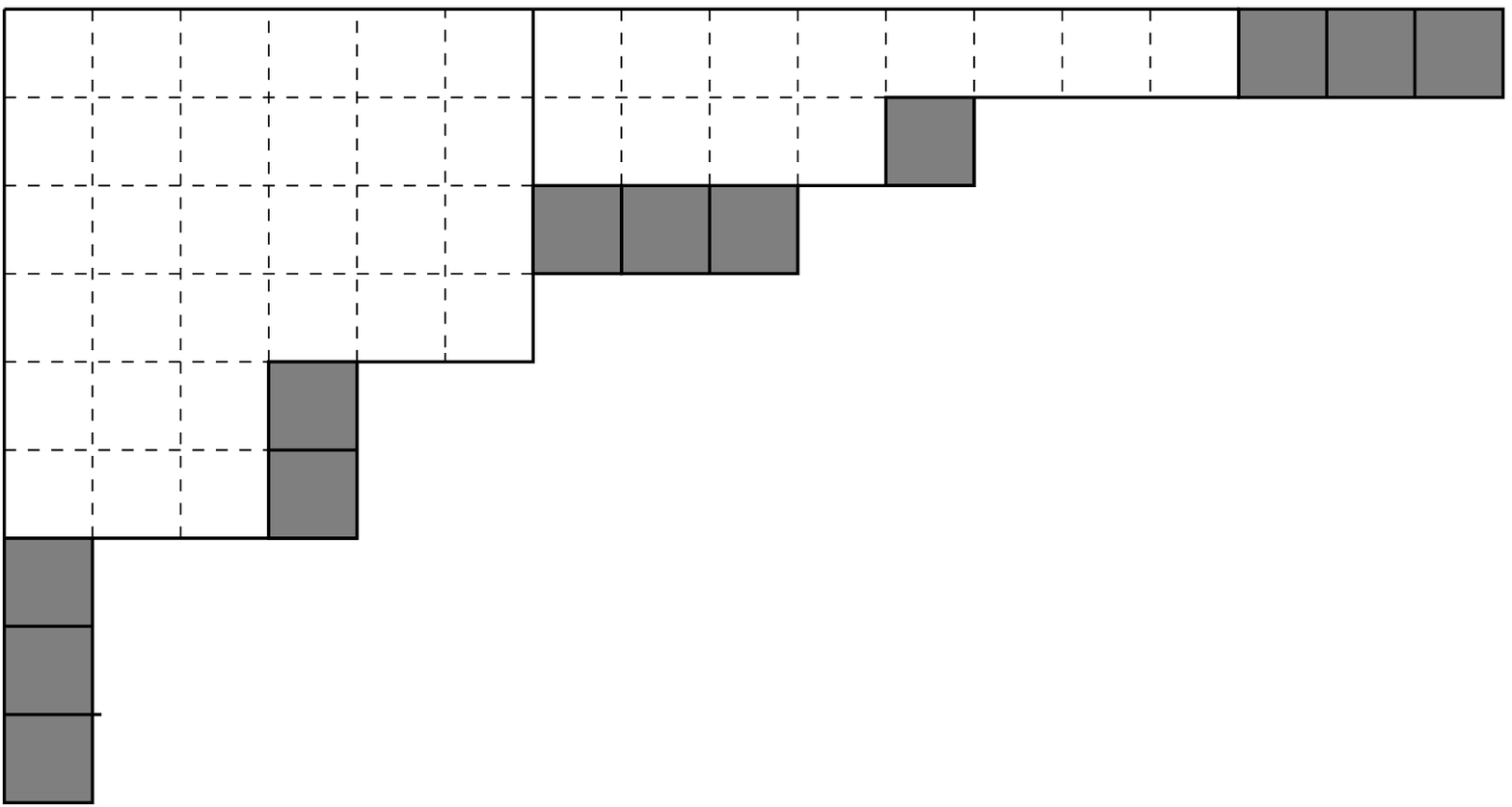}
\qquad \
\includegraphics[scale=0.32]{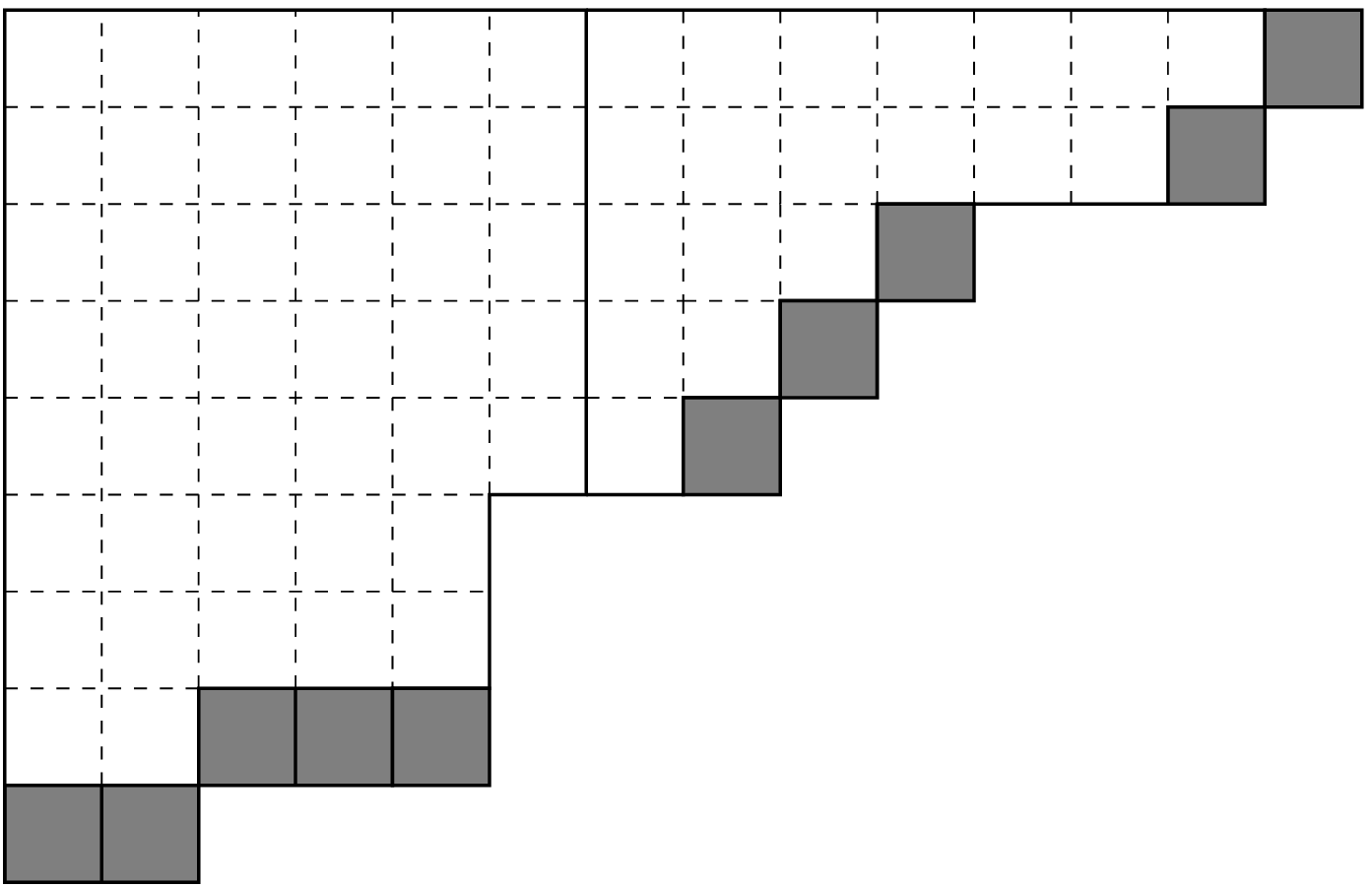}
\]
\end{example}
\smallskip

The $z$- and $\ov{z}$-strips are characterized among all
$k$-horizontal strips as follows.

\begin{prop}
\label{khoriz}
A $k$-horizontal strip $\la/\mu$ is a $z$-strip
(resp.\ $\ov{z}$-strip) if and only if the left boxes in $\la/\mu$
form a vertical strip (resp.\ horizontal strip), and no two boxes in
$\la/\mu$ are $k'$-related (resp.\ no two right boxes in $\la/\mu$ are
in the same row).
\end{prop}
\begin{proof}
A signed permutation is decreasing down to $0$ (resp.\ increasing up
from $0$) if and only if it has no reduced word which contains $i-1,
i$ (resp.\ $i, i-1$) as a subword for some $i\geq 1$. Observe that in
any reduced factorization $w_{\la/\mu}=us_{i-1}s_iv$
(resp.\ $w_{\la/\mu}=us_is_{i-1}v$), the elements $u$ and $v$ are also
skew, and associated to $k$-horizontal strips which are substrips of
$\la/\mu$. Therefore we may assume that $w_{\la/\mu}\in \{s_{i-1}s_i,
s_is_{i-1}\}$ for some $i$ and study the associated skew diagram
$\la/\mu$. For $i\geq 2$, the argument is the same as in the proof of
\cite[Prop.\ 1]{T6}, so it remains to examine the case when $i=1$.

If $w_{\la/\mu}=s_0s_1$ then $w_\mu$ and $w_\la$ have the form
\[
w_\mu = (\cdots | \cdots \ov{1}2 \cdots),  \ \, 
w_\la = (\cdots | \cdots \ov{2}\,\ov{1} \cdots)
\]
so that $\la/\mu$ has two right boxes which are $k'$-related, or 
\[
w_\mu = (1 \cdots | \cdots 2 \cdots),  \ \, 
w_\la = (2 \cdots | \cdots \ov{1} \cdots)
\]
so that $\la/\mu$ has a left box and a right box which are
$k'$-related (and in the same row). Here the vertical line $|$
lies between positions $k$ and $k+1$. On the other hand,
if $w_{\la/\mu}=s_1s_0$ then $w_\mu$ and $w_\la$ have the form
\[
w_\mu = (2 \cdots | \cdots 1 \cdots)
\ \, \text{or}
\ \, w_\mu = (\cdots | \cdots 12 \cdots),  \ \, 
w_\la = (\cdots | \cdots \ov{2} \cdots)
\]
so that $\la/\mu$ has two right boxes which are in the same row
(and in columns $k+1$ and $k+2$). The converse statements are proved
by using the correspondence between $\nu$ and $w_\nu$ given in Section
\ref{typesBC}. This completes the proof.
\end{proof}

\begin{defn}
\label{krook}
For any pair $\la,\mu$ of $k$-strict partitions with $\mu\subset \la$,
we say that the skew Young diagram $\la/\mu$ is a {\em $k$-rook strip}
if it contains no two boxes in the same row or column, and no two
right boxes which are $k$-related.
\end{defn}

Observe that the skew diagram $\la/\mu$ is a $k$-rook strip if and
only if $\la\ssm b$ is the diagram of a $k$-strict partition for each
box $b$ in $\la/\mu$. The direct analogue of Proposition
\ref{skewfactorA} does not hold in types B, C, or D, as the next
example illustrates.

\begin{example}
Let $\la=\nu:=(4,1)$, $\rho:=2$, and $\mu:=\emptyset$, so that
$\mu\subset\rho\subset\nu\subset\la$, $w_\la=s_1s_2s_1s_0s_1$,
and $w_\rho=s_0s_1$. If $u:=s_1s_2s_1$ and $v:=w_\la$, then $w_{\la/\mu}=u\di v$ and we
have $u=w_{\la/\rho}$ and $v=w_{\nu/\mu}$. However, $\nu/\rho=(4,1)/2$ is not a
$k$-rook strip. Notice that neither $u$ nor $v$ is fully commutative in the sense of
\cite{St}.
\end{example}

In place of Proposition \ref{skewfactorA} and Corollary \ref{Afact}, we have the
following weaker results.

\begin{prop}
\label{skewfactorC}
Suppose that $w_{\la/\mu}=u\di v$ for some skew element $w_{\la/\mu}$
of $W_\infty$. Then there exist $k$-strict partitions $\rho$ and $\nu$
such that $\mu\subset\rho\subset\la$ and $\mu\subset\nu\subset\la$
are compatible triples, $u=w_{\la/\rho}$, and $v=w_{\nu/\mu}$.
\end{prop}
\begin{proof}
Suppose that $a_1 \cdots a_r$ is a reduced word for $u$, so that
$w_{\la/\mu} = s_{a_1}\di\cdots \di s_{a_r} \di v$. Since $s_i\di v$
is equal to $v$ or to the reduced product $s_iv$ for any $i$, it is
clear that there exists a reduced factorization $w_{\la/\mu}= w'v$ for
some $w'$, and hence that $v=w_{\nu/\mu}$ for some $k$-strict
partition $\nu$ with $\mu\subset\nu\subset\la$. Moreover, we have
$w'=w_{\la/\nu}$ and the pairs $(\la,\nu)$ and $(\nu,\mu)$ are
compatible.  Similarly, we see that there exists a reduced
factorization $w_{\la/\mu}= uw''$ for some $w''$, and hence that
$u=w_{\la/\rho}$ for some $k$-strict partition $\rho$ with
$\mu\subset\rho\subset\la$, such that the pairs $(\la,\rho)$ and
$(\rho,\mu)$ are compatible.
\end{proof}

\begin{cor}
\label{Cfact}
For the skew element $w_{\la/\mu}\in W_\infty$, we have
\[
\CG_{w_{\la/\mu}}(Z;X,Y) = \sum_{\rho,\nu} \CG_{w_{\nu/\mu}}(Z;X)\AG_{w_{\la/\rho}^{-1}}(Y)
\]
summed over all $k$-strict partitions $\rho$ and $\nu$ such that
$\mu\subset\rho\subset\la$ and $\mu\subset\nu\subset\la$ are
compatible triples, $\ell_k(\rho)=\ell_k(\la)$, and
$w_{\la/\mu}=w_{\la/\rho}\di w_{\nu/\mu}$.
\end{cor}
\begin{proof}
The definition (\ref{dbleC}) of $\CG_w$ implies that
\[
\CG_w(Z;X,Y) = \sum_{u\di v =w} \CG_v(Z;X)\AG_{u^{-1}}(Y)
\]
summed over all Hecke factorizations $u\di v$ of $w$ with $u\in S_\infty$.
The result follows immediately from this, Proposition \ref{skewfactorC}, and the
fact that $w_{\la/\rho}\in S_{\infty}$ if and only if $\ell_k(\rho)=\ell_k(\la)$.
\end{proof}

Suppose that $\la$ is a $k$-strict partition and $i\geq 0$ is such
that $\ell(s_iw_\la)=\ell(w_\la)-1$. Then $s_iw_\la=w_\mu$ for a $k$-strict
partition $\mu\subset\la$, so that $\la/\mu$ is a $k$-horizontal strip
consisting of a single box $b$ of $\la$. We then say that $b$ is a
{\em removable box} of $\la$. It follows that a box $b$ of $\la$ is a
removable box of $\la$ if and only if (i) $\la\ssm b$ is a
partition, and (ii) if $b$ is a left box of $\la$, then $b$ is not
$k'$-related to two right boxes of $\la$ (including boxes in row zero)
which are both bottom boxes in their column.

\begin{example}
Let $k:=1$ and $\la:=(3,1)$, so that $w_\la = s_1s_0s_2s_1$. Then the
box of $\la$ in the third column is removable, but the box in the
second row is not removable. The diagram $(3,1)/3$ is a $1$-rook strip but
$((3,1),3)$ is not a compatible pair.
\end{example}

Similarly, suppose that $(\la,\mu)$ is a compatible pair of $k$-strict
partitions and $i\geq 0$ is such that
$\ell(s_iw_{\la/\mu})=\ell(w_{\la/\mu})-1$. Then
$s_iw_{\la/\mu}=w_{\nu/\mu}$ for a $k$-strict partition $\nu$ with
$\mu\subset\nu\subset\la$, so that $\la/\nu$ is a $k$-horizontal strip
consisting of a single box $b$ of $\la$, which is not in $\mu$. We
then say that $b$ is a {\em removable box} of $\lambda/\mu$.

\begin{example}
Let $k:=1$, $\la:=(4,1)$, and $\mu:=3$, so
$w_\la=s_1s_2s_1s_0s_1=s_2s_1s_2s_0s_1$, $w_\mu=s_1s_0s_1$, and
$w_{\la/\mu}=s_1s_2$. Then $\la$ has two removable boxes which are not
in $\mu$, but only the box $b:=(4,1)/4$ in the second row is a
removable box of $\la/\mu$.
\end{example}

\begin{prop}
\label{Cprop}
Suppose that $(\la,\mu)$ is a compatible pair of $k$-strict partitions,
and $b$ is a box of $\la$ which is not in $\mu$. Then $b$ is a removable
box of $\la/\mu$ if and only if $b$ is a removable box of $\la$ and
$(\la\ssm b,\mu)$ is a compatible pair.
\end{prop}
\begin{proof}
Suppose that $b$ is a removable box of $\la/\mu$ and let $i\geq 0$ be
such that $s_iw_{\la/\mu} = w_{\nu/\mu}$, with $\nu:=\la\ssm b$. Then
$(\nu,\mu)$ is a compatible pair and $s_iw_{\la}=s_iw_{\la/\mu}w_\mu =
w_{\nu/\mu}w_\mu = w_\nu$, so $b$ is a removable box of $\la$.
Conversely, suppose that $b$ is a removable box of $\la$, so that
$s_iw_\la=w_\nu$ with $\nu:=\la\ssm b$, and assume that $(\nu,\mu)$ is
a compatible pair. Since $(\la,\nu)$ is also compatible, we have
$w_{\la/\mu}=w_{\la/\nu}w_{\nu/\mu}=s_iw_{\nu/\mu}$, hence
$s_iw_{\la/\mu}=w_{\nu/\mu}$ and so $b$ is a removable box of
$\la/\mu$.
\end{proof}

The following proposition and its converse are the reason we can
achieve tableau formulas for Grothendieck polynomials in types B, C,
and D beyond the fully commutative Grassmannian examples of \cite{GK,
  IN}. The point is that although Proposition \ref{skewfactorA} fails
in these Lie types, an analogous result is true when the left factor
$u$ has a reduced word which is decreasing or increasing.

\begin{prop}
\label{skewfactC}
Suppose that $u$ is decreasing down to $0$ or increasing up from $0$
and $w_{\la/\mu}=u\di v$ for some skew element $w_{\la/\mu}$. Then
there exist $k$-strict partitions $\rho\subset\nu$ such that
$\mu\subset\rho\subset\la$ and $\mu\subset\nu\subset\la$ are
compatible triples, $\nu/\rho$ is a $k$-rook strip, $u=w_{\la/\rho}$,
and $v=w_{\nu/\mu}$.
\end{prop}
\begin{proof}
Proposition \ref{skewfactorC} implies that $v=w_{\nu/\mu}$ for some
$k$-strict partition $\nu$ such that $\mu\subset\nu\subset\la$ is a
compatible triple. We argue by induction on $\ell(u)$. If $\ell(u)=1$,
then $u=s_i$ is a simple reflection. If the product $s_iw_{\nu/\mu}$
is reduced, then $u=w_{\la/\nu}$ and we can take $\rho$ to be equal to
$\nu$. If the product $s_iw_{\nu/\mu}$ is not reduced, then $\nu=\la$
and we have $s_iw_{\la/\mu}=w_{\rho/\mu}$ for some $k$-strict
partition $\rho\subset\la$, where $b:=\la/\rho$ is a removable box of
$\la/\mu$. We are therefore done in both cases.

For the inductive step, we may assume that $u=s_iu'$ with
$\ell(u)=\ell(u')+1$ and $u$ is decreasing down to $0$. We have $u'\di
w_{\nu/\mu} = w_{\zeta/\mu}$ for some $k$-strict partition $\zeta$
with $\mu\subset\zeta\subset\la$.  The inductive hypothesis implies
that $u'=w_{\zeta/\rho}$ for some $k$-strict partition
$\rho\subset\nu$ such that $\mu\subset\rho\subset\zeta$ and
$\mu\subset\nu\subset\zeta$ are compatible triples and $\nu/\rho$ is a
$k$-rook strip.  If $s_iw_\zeta=w_\la$ then $(\la,\zeta)$ is a
compatible pair. Hence $(\la,\rho)$ is a compatible pair,
$u=s_iw_{\zeta/\rho}=w_{\la/\rho}$, and we are done.

Otherwise, we have $\zeta=\la$ and so $u'=w_{\la/\rho}$, while
$s_iw_\la = w_{\la\ssm b}$ for some removable box $b$ of $\la$ which
is not in $\mu$.  Now there is a reduced word $a_r\cdots a_1$ for
$w_{\la/\rho}$ such that $i>a_r>\cdots>a_1\geq 0$. Therefore
$\la/\rho$ is a $z$-strip and, using Propositions \ref{khoriz} and
\ref{skewfactorC}, we see that $u=w_{\la/\wt{\rho}}$ for some
$z$-strip $\la/\wt{\rho}$ with $|\la/\wt{\rho}|=r+1$ that contains
$\la/\rho$ as a substrip. As
$\ell(s_iw_{\la/\wt{\rho}})=\ell(w_{\la/\wt{\rho}})-1$, we deduce that
$b$ is a removable box of $\la/\wt{\rho}$ which is not in $\la/\rho$.
Hence $b$ is a box of $\rho$, $\wt{\rho}=\rho\ssm b$, and the equation
$w_{\la/\mu}=w_{\la/\wt{\rho}}\di v$ implies that
$\mu\subset\wt{\rho}\subset\la$ is a compatible triple. Since $\la\ssm
b$ is a $k$-strict partition, it follows that $\nu\ssm b$ is a
$k$-strict partition. As $\nu/\rho$ is a $k$-rook strip, we conclude
that $\nu/\wt{\rho} = (\nu/\rho)\cup b$ is also a $k$-rook strip.
Finally, the proof in case $u$ is increasing up from $0$ is similar.
\end{proof}

\begin{remark}
In the situation of Proposition \ref{skewfactC}, if 
$u$ is decreasing down to $0$ (resp.\ increasing up from $0$),
then $\la/\rho$ and $\la/\nu$ are both $z$-strips (resp.\
$\ov{z}$-strips). To see this for $\la/\nu$, note that if
$a_r\cdots a_1$ is a decreasing (resp.\ increasing) reduced word of
$u$, then $s_{a_r}\di \cdots \di s_{a_1}\di w_{\nu/\mu} = w_{\la/\mu}$ implies that
$s_{a_{j_p}}\cdots s_{a_{j_1}} w_{\nu/\mu}=w_{\la/\mu}$ for some subword $a_{j_p}\cdots a_{j_1}$ of
$a_r\cdots a_1$, and hence $w_{\la/\nu} = s_{a_{j_p}}\cdots s_{a_{j_1}}$.
\end{remark}

\begin{example}
Let $k:=1$, $\nu:=(3,1)$, and $\la:=(4,1)$, so that
$w_\nu=3\ov{2}1=s_1s_2s_0s_1$ and $w_\la=2\ov{3}1=s_2w_\nu$.  Let
$u:=s_1s_2$ and observe that $w_\la=u\di w_\nu$. We then have
$u=w_{\la/\rho}$ for $\rho:=3\subset \nu$. Note that $b:=\nu/\rho$ 
is a removable box of $\la$, but not a removable box of $\nu$. 
This shows that $(\nu,\rho)$ need not be a compatible pair in
Proposition \ref{skewfactC}. Next, let 
 $v:=s_2s_1$ and observe that again $w_\la=v\di w_\nu$. We have
$v=w_{\la/\wt{\rho}}$ for $\wt{\rho}:=(2,1)\subset \nu$. In this case
$\wt{b}:=\nu/\wt{\rho}$ is a removable box of $\nu$, but not a
removable box of $\la$.
\end{example}

The next result is a converse to Proposition \ref{skewfactC}.

\begin{prop}
\label{skewCconverse}
Suppose that $\mu\subset\rho\subset\nu\subset\la$ are $k$-strict
partitions such that $\mu\subset\rho\subset\la$ and
$\mu\subset\nu\subset\la$ are compatible triples, $\la/\rho$ is a
$z$-strip or a $\ov{z}$-strip, and $\nu/\rho$ is a
$k$-rook strip. Then we have a Hecke factorization
$w_{\la/\mu}=w_{\la/\rho}\di w_{\nu/\mu}$.
\end{prop}
\begin{proof}
Suppose that $\la/\rho$ is a $\ov{z}$-strip and $a_r\cdots a_1$ is a
reduced word for $w_{\la/\rho}$ with $r=|\la/\rho|$ and $0\leq a_r <
\cdots < a_1$. Since $(\rho,\mu)$ is a compatible pair, we can define
a $k$-strict partition $\zeta_i\supset\rho$ for each $i\in [1,r]$ by
the equation $s_{a_i}\cdots s_{a_1} w_{\rho/\mu}=w_{\zeta_i/\mu}$, and
let $b_i$ denote the box $\zeta_i\ssm\zeta_{i-1}$.  Define
$\ov{\zeta}_i:=\zeta_i\cup\nu$, $\zeta_0:=\rho$, $\ov{\zeta}_0:=\nu$,
and set $w^{(i)}:=s_{a_i}\di \cdots \di s_{a_1}\di w_{\nu/\mu}$. As
$\nu/\rho$ is a $k$-rook strip, it follows that $\ov{\zeta}_i$ is a
$k$-strict partition. We will prove by induction on $i$ that
$(\ov{\zeta}_i,\nu)$ is a compatible pair and $w^{(i)} =
w_{\ov{\zeta}_i/\mu}$ for each $i\in [1,r]$.  Assume therefore that
the pair $(\ov{\zeta}_{i-1},\nu)$ is compatible, $w^{(i-1)} =
w_{\ov{\zeta}_{i-1}/\mu}$, and consider the product $s_{a_i}\di
w_{\ov{\zeta}_{i-1}/\mu} = w^{(i)}$.

Suppose first that $b_i\notin \nu$, hence
$b_i=\ov{\zeta}_i\ssm\ov{\zeta}_{i-1}$. Since $\la/\rho$ is a
$k$-horizontal strip and $(\la,\nu)$ is compatible pair, we deduce
that $\la/\nu$ is also a $k$-horizontal strip. Proposition
\ref{khoriz} then implies that $\la/\nu$ is a $\ov{z}$-strip.
Moreover, $a_r\cdots a_1$ has a subword $a_{j_p}\cdots a_{j_1}$ which
is a reduced word for $w_{\la/\nu}$. Since we have $i\in
\{j_1,\ldots,j_p\}$, we deduce that the pair
$(\ov{\zeta}_i,\ov{\zeta}_{i-1})$ is compatible. As the pair
$(\ov{\zeta}_{i-1},\nu)$ is also compatible, it follows that 
$(\ov{\zeta}_i,\nu)$ is compatible. Therefore $(\ov{\zeta}_i,\mu)$ is
compatible and we have $s_{a_i}\di w_{\ov{\zeta}_{i-1}/\mu} = s_{a_i}
w_{\ov{\zeta}_{i-1}/\mu}= w_{\ov{\zeta}_i/\mu}$, as required.

Next, suppose that $b_i$ is a box of the $k$-rook
strip $\nu/\rho$, hence $\ov{\zeta}_i=\ov{\zeta}_{i-1}$ and the
pair $(\ov{\zeta}_i,\nu)$ is compatible. Using
Proposition \ref{khoriz} again, we see that $(\ov{\zeta}_i,\rho)$ is a
compatible pair and $a_r\cdots a_1$ has a subword
which includes $a_i$ and is a reduced word for $w_{\ov{\zeta}_i/\rho}$.
This proves that $b_i$ is a removable box of $\ov{\zeta}_i/\rho$ and
hence $s_{a_i}\di w_{\ov{\zeta}_i/\rho} =w_{\ov{\zeta}_i/\rho}$.
Since $w_{\ov{\zeta}_i/\mu}=w_{\ov{\zeta}_i/\rho}w_{\rho/\mu}$ we deduce 
that 
$s_{a_i}\di w_{\ov{\zeta}_{i-1}/\mu} = s_{a_i}\di
w_{\ov{\zeta}_i/\mu} = w_{\ov{\zeta}_i/\mu}$.  As
$\ov{\zeta}_r=\la$, we conclude that $w_{\la/\rho}\di w_{\nu/\mu}= w^{(r)} =
w_{\la/\mu}$.  Finally, the argument when $\la/\rho$ is a $z$-strip is
similar.
\end{proof}

\begin{example}
The hypotheses that $(\la,\nu)$ and $(\rho,\mu)$ are compatible pairs
in Proposition \ref{skewCconverse} are necessary. For instance, let
$k:=1$, $\mu:=\emptyset$, $\rho:=2$, $\nu:=3$, and $\la:=(3,1)$. Then
$\mu\subset\rho\subset\la$ is a compatible triple, $(\nu,\mu)$ is a
compatible pair, $w_{\la/\rho}=s_1s_2$ is a $\ov{z}$-strip, and
$\nu/\rho$ is a $1$-rook strip. However, we have
\[
w_{\la/\rho}\di w_\nu = (s_1s_2)\di (s_1s_0s_1) = s_1s_2s_1s_0s_1 \neq w_\la.
\]
Moreover, if $k=1$, $\mu:=3$, $\rho:=(3,1)$, and $\nu=\la:=(4,1)$,
then $\mu\subset\nu\subset\la$ is a compatible triple, $(\la,\rho)$ is
a compatible pair, and $\nu/\rho$ is a $1$-rook strip.  However, we
compute that
\[
w_{\la/\rho}\di w_{\nu/\mu} = s_2 \di (s_1s_2) = s_2s_1s_2\neq w_{\la/\mu}.
\]
\end{example}

\begin{defn}
\label{svktab}
A {\em set-valued $k$-tableau} $T$ of shape $\la/\mu$ is a pair of
sequences of $k$-strict partitions
\[
\mu = \la^0 \subset \la^1 \subset \cdots \subset \la^{2r} =\la
\quad \mathrm{and} \quad
\mu = \nu^0, \nu^1, \ldots, \nu^{2r} = \la
\]
with the first one compatible,
such that (i) $\mu\subset \nu^i\subset\la^i$, $(\nu^i,\mu)$ is a compatible pair,
and $\la^i/\nu^i$ is a $k$-rook
strip for each $i$, and (ii) $\la^i/\nu^{i-1}$ is a $\ov{z}$-strip if
$i$ is odd and a $z$-strip if $i$ is even, for $1\leq i\leq 2r$.  We
represent $T$ by a filling of the boxes in $\la/\mu$ with finite
nonempty subsets of barred and unbarred positive integers such that
for each $h\in [1,r]$, the boxes in $T$ which contain the entry
$\ov{h}$ (resp.\ $h$) form the skew diagram $\la^{2h-1}/\nu^{2h-2}$
(resp.\ $\la^{2h}/\nu^{2h-1}$).  For any set-valued $k$-tableau $T$ we
let $|T|$ be the total number of barred and unbarred positive integers
which appear in $T$, including their multiplicities, and set
$z^T:=\prod_h z_h^{m_h}$, where $m_h$ denotes the number of times that
$\ov{h}$ or $h$ appears in $T$.
\end{defn}

\begin{thm}
\label{FCthm}
Let $w=w_{\la/\mu}$ be the skew element associated to the compatible
pair $(\la,\mu)$ of $k$-strict partitions. Then we have
\begin{equation}
\label{Ftabeq}
F^C_w(Z) = \sum_T \beta^{|T|-|\la/\mu|}\, z^T
\end{equation}
where the sum is over all set-valued $k$-tableau $T$ of shape $\la/\mu$.
\end{thm}
\begin{proof}
We deduce from the definition of $F^C_w$ that
\begin{equation}
\label{Ftabeq2}
F^C_w(Z) = \sum_{w=\s_{2r} \di\cdots\di \s_1}
\be^{\sum_i \ell(\s_i)-\ell(w)}\prod_j z_j^{\ell(\s_{2j-1})+\ell(\s_{2j})}
\end{equation}
where the inner sum is over all Hecke factorizations $\s_{2r}\di \cdots
\di \s_1$ of $w$, for varying $r$, such that $\s_{2j-1}$ is
increasing up from $0$ and $\s_{2j}$ is decreasing down to $0$ for all
$j\in [1,r]$. Propositions \ref{skewfactC} and \ref{skewCconverse}
imply that any such factorization $w=\s_{2r}\di \cdots \di \s_1$
corresponds to a pair of sequences of $k$-strict partitions
\[
\mu = \la^0 \subset \la^1 \subset \cdots \subset \la^{2r} =\la \quad \mathrm{and} \quad
\mu = \nu^0, \nu^1, \ldots, \nu^{2r} = \la 
\]
as in Definition \ref{svktab}, with $\s_i= w_{\la^i/\nu^{i-1}}$ for each $i\in [1,2r]$.
Moreover, if $T$ is the associated set-valued $k$-tableau, then
\[
|T|= \sum_{i=1}^{2r}|\la^i/\nu^{i-1}| = \sum_{i=1}^{2r}\ell(\s_i)
\quad \text{and} \quad
z^T=\prod_{j=1}^r z_j^{\ell(\s_{2j-1})+\ell(\s_{2j})}
\]
while $\ell(w)=|\la/\mu|$. Equation (\ref{Ftabeq}) therefore follows immediately
from (\ref{Ftabeq2}).
\end{proof}

\begin{example}
Let $k:=1$ and $\la:=(5,3,1)$ so that $w_\la=3\,\ov{4}\,\ov{2}\,1 =
s_1s_3s_2s_0s_3s_1s_0s_2s_1$. The Hecke factorization
\[
(s_3s_1s_0) \di 1 \di (s_3s_2)\di 1 \di (s_2s_0) \di (s_2s_3)
\di (s_3s_1s_0) \di (s_0s_2)\di (s_2s_1) \di 1
\]
of $w_\la$ as a product of $2r=10$ factors $\s_i$ in $W_4$ corresponds to the following
set-valued $1$-tableau $T$ of shape $\la$:
\[
\begin{array}{ccccc} 1 & \{\ov{2},2\} & 2 & \{\ov{3},3,4\} & \{4,5\} \\
\{1,\ov{2}\} & \{3,5\} & 5 && \\ \{2,\ov{3}\} &&&& \end{array}.
\]
Furthermore, we have $z^T=z_1^2z_2^5z_3^4z_4^2z_5^3$.
\end{example}

\begin{example}
\label{exCtab}
Let $(\la,\mu)$ be a compatible pair of typed $k$-strict partitions. A
{\em barred $k$-tableau} $T$ of shape $\la/\mu$ is a sequence of
$k$-strict partitions
\[
\mu = \la^0 \subset \la^1 \subset \cdots \subset \la^{2r} =\la
\]
such that $\la^i/\la^{i-1}$ is a $\ov{z}$-strip if $i$ is odd and a
$z$-strip if $i$ is even, for $1\leq i\leq 2r$.  We represent $T$ by a
filling of the boxes in $\la/\mu$ with barred and unbarred positive
integers such that for each $h\in [1,r]$, the boxes in $T$ with entry
$\ov{h}$ (resp.\ $h$) form the skew diagram $\la^{2h-1}/\la^{2h-2}$
(resp.\ $\la^{2h}/\la^{2h-1}$).

If we set $\beta:=0$, then the $K$-theoretic Stanley function
$F^C_w(Z)$ specializes to the type C Stanley function $F_w(Z)$,
introduced and studied in \cite{BH, FK3, La}. If $w:=w_{\la/\mu}$
then Theorem \ref{FCthm} gives
\begin{equation}
\label{Feq}
F_w(Z) = \sum_T z^T
\end{equation}
summed over all barred $k$-tableaux $T$ of shape $\la/\mu$.
The barred $k$-tableaux and equation (\ref{Feq}) refine the
$k$-tableaux of \cite[Def.\ 2]{T1} and \cite[Eqn.\ (7)]{T6} in the
same way that the marked shifted tableaux of Sagan and Worley
\cite{Sa, W} and \cite[III.(8.16$'$)]{Mac} refine the shifted
tableaux and \cite[III.(8.16)]{Mac}.

For example, let $k=1$, $\la=(3,1)$, $Z_2:=(1,2)$, and {\bf Q}$_{12}$ denote
the alphabet $\{\ov{1}<1<\ov{2}<2\}$. Following \cite[Example 7]{T1}, there are
four $1$-tableaux $T$ of shape $\la$ with entries in $\{1,2\}$, each with
multiplicity $2^{n(T)}=4$:
\begin{gather*}
\begin{array}{l} 1\,2\,2 \\ 2 \end{array} \ \ \ \
\begin{array}{l} 1\,1\,2 \\ 2 \end{array}  \ \ \ \
\begin{array}{l} 1\,2\,2 \\ 1 \end{array}  \ \ \ \
\begin{array}{l} 1\,1\,2 \\ 1 \end{array}.
\end{gather*}
These correspond to the following 16 barred $1$-tableaux with entries
in {\bf Q}$_{12}$:
\begin{gather*}
\begin{array}{l} \ov{1}\,2\,2 \\ \ov{2} \end{array}  \ \ \ 
\begin{array}{l} \ov{1}\,\ov{2}\,2 \\ \ov{2} \end{array}  \ \ \
\begin{array}{l} 1\,2\,2 \\ \ov{2} \end{array}  \ \ \
\begin{array}{l} 1\,\ov{2}\,2 \\ \ov{2} \end{array}  \ \ \
\begin{array}{l} \ov{1}\,\ov{1}\,\ov{2} \\ \ov{2} \end{array}  \ \ \ 
\begin{array}{l} \ov{1}\,\ov{1}\,2 \\ \ov{2} \end{array}  \ \ \
\begin{array}{l} \ov{1}\,1\,\ov{2} \\ \ov{2} \end{array}  \ \ \ 
\begin{array}{l} \ov{1}\,1\,2 \\ \ov{2} \end{array}  
\end{gather*}
\begin{gather*}
\begin{array}{l} \ov{1}\,\ov{2}\,2 \\ 1 \end{array}  \ \ \ 
\begin{array}{l} \ov{1}\,2\,2 \\ 1 \end{array}  \ \ \
\begin{array}{l} 1\,\ov{2}\,2 \\ 1 \end{array}  \ \ \ 
\begin{array}{l} 1\,2\,2 \\ 1 \end{array}  \ \ \
\begin{array}{l} \ov{1}\,\ov{1}\,\ov{2} \\ 1 \end{array}  \ \ \ 
\begin{array}{l} \ov{1}\,\ov{1}\,2 \\ 1 \end{array}  \ \ \
\begin{array}{l} \ov{1}\,1\,\ov{2} \\ 1 \end{array}  \ \ \ 
\begin{array}{l} \ov{1}\,1\,2 \\ 1 \end{array}.
\end{gather*}
We deduce that $F_{w_\la}(Z_2)=4z_1^3z_2 + 8 z_1^2z_2^2 + 4 z_1z_2^3$.
\end{example}

Let {\bf Q} denote the ordered alphabet
\[
(n-1)'<\cdots 2'<1'< \ov{1} < 1 < \ov{2} < 2 < \cdots <1''<2''<\cdots<(n-1)''. 
\]
The symbols $\ov{1},\ov{2},\ldots$ (resp.\ $1,2,\ldots$) are 
{\em barred integers} (resp.\ {\em unbarred integers}).

\begin{defn}
\label{Ctrit}
Consider a pair $U$ of sequences of $k$-strict partitions
\[
\mu = \la^0 \subset \la^1 \subset \cdots \subset \la^{2n+2r-2} =\la
\quad \mathrm{and} \quad
\mu = \nu^0, \nu^1, \ldots, \nu^{2n+2r-2} = \la 
\]
with the first one compatible,
such that (i) $\mu\subset\nu^i\subset\la^i$, $(\nu^i,\mu)$ is a compatible pair,
and $\la^i/\nu^i$ is a $k$-rook
strip for each $i$, and (ii) $\la^i/\nu^{i-1}$ is an
$x$-strip for each $i\leq n-1$, a $y$-strip for each $i\geq n+2r$,
and a $\ov{z}$-strip (resp.\ $z$-strip) if $i = n-1+j$ for some odd
(resp.\ even) $j\in [1,2r]$. We represent $U$ by a filling of the
boxes in $\la/\mu$ by including the entry $(n-i)'$ in each box of
$\la^i/\nu^{i-1}$ for $1\leq i\leq n-1$, the entry $(i-n-2r+1)''$ in
each box of $\la^i/\nu^{i-1}$ for $n+2r\leq i\leq 2n+2r-2$, and the
entry $\ov{j}$ (resp.\ $j$) in each box of $\la^{n-1+j}/\nu^{n-2+j}$
for all odd (resp.\ even) integers $j\in [1,2r]$. We say that $U$ is a
{\em set-valued $k$-tritableau} of shape $\la/\mu$ if for $1\leq i
\leq \ell_k(\mu)$ (resp.\ $1\leq i \leq \ell_k(\la)$) and $1\leq j
\leq k$, the entries of $U$ in row $i$ are $\geq (\mu_i-k)'$
(resp.\ $\leq (\la_i-k-1)''$) and the entries in column $k+1-j$ lie in
the interval $[(w_\mu(j))', (w_\la(j)-1)'']$.  We let $|U|$ be the
total number of symbols of {\bf Q} which appear in $U$, including
their multiplicities, and define
\[
(xyz)^U:= \prod_ix_i^{n'_i}\prod_i y_i^{n''_i}\prod_j z_j^{n_j} 
\]
where $n'_i$ and $n''_i$ denote the number of times that $i'$ and
$i''$ appear in $U$, respectively, for each $i\in [1,n-1]$, and $n_j$
denotes the number of times that $j$ or $\ov{j}$ appears in $U$, for
each $j \geq 1$.
\end{defn}

\begin{thm}
\label{Cskew} 
For the skew element $w:=w_{\la/\mu}$ of $W_n$, we have
\begin{equation}
\label{CGteq}
\CG_w(Z;X,Y)=\sum_U \be^{|U|-|\la/\mu|}\,(xyz)^U
\end{equation}
summed over all set-valued $k$-tritableaux $U$ of shape $\la/\mu$.
\end{thm}
\begin{proof}
It follows from formula (\ref{dbleC}) that we have 
\[
\CG_w(Z;X,Y) 
= \sum \be^{\ell(u,v,\sigma,w)} y_{n-1}^{\ell(v_{n-1})}\cdots y_1^{\ell(v_1)}
x_1^{\ell(u_1)}\cdots x_{n-1}^{\ell(u_{n-1})}\prod_j z_j^{\ell(\sigma_{2j-1})+\ell(\s_{2j})}
\]
where the sum is over all Hecke factorizations
\[
v_{n-1}\di\cdots
\di v_1\di \sigma_{2r}\di\cdots\di\sigma_1 \di u_1\di\cdots \di u_{n-1}
\]
of $w$, for varying $r\geq 0$, such that $v_p\in S_n$ is increasing up
from $p$ and $u_p\in S_n$ is decreasing down to $p$ for each $p$,
$\sigma_i$ is increasing up from $0$ (resp.\ decreasing down to $0$)
for all odd $i$ (resp.\ all even $i$), and
$\ell(u,v,\sigma,w):=\sum_p\ell(u_p)+\sum_p\ell(v_p)
+\sum_i\ell(\sigma_i)- \ell(w)$.  We deduce from Propositions
\ref{skewfactC} and \ref{skewCconverse} that such factorizations
correspond to pairs of sequences of $k$-strict partitions
\[
\mu = \la^0 \subset \la^1 \subset \cdots \subset \la^{2n+2r-2} =\la
\quad \mathrm{and} \quad 
\mu = \nu^0, \nu^1, \ldots, \nu^{2n+2r-2} = \la 
\]
and an associated filling $U$ of the the boxes in $\la/\mu$ as in
Definition \ref{Ctrit}. The required bounds on the entries of $U$ are
established exactly as in the proof of \cite[Thm.\ 2]{T6}, showing
that $U$ is a set-valued $k$-tritableau of shape $\la/\mu$ such that
\[
(xyz)^U=y_{n-1}^{\ell(v_{n-1})}\cdots y_1^{\ell(v_1)}x_1^{\ell(u_1)}\cdots
x_{n-1}^{\ell(u_{n-1})}\prod_j z_j^{\ell(\sigma_{2j-1})+\ell(\s_{2j})}
\]
and $|U|=\sum_p\ell(u_p)+\sum_p\ell(v_p)+\sum_i\ell(\sigma_i)$. Since the sum in equation
(\ref{CGteq}) is over all such $U$, the result follows.
\end{proof}

\begin{example}
Let $k:=1$ and $\la:=(4,2)$ so that $w_\la=2\,\ov{3}\,\ov{1} =s_2s_0s_1s_2s_0s_1$,
and consider the alphabet ${\bf Q}_{ex}:=\{1'<\ov{1}<1<\ov{2}<1''<2''\}$.
The set-valued $1$-tritableau
\[
\begin{array}{cccc} 1' & \{\ov{1},1\} & 1 & \{1,\ov{2},1''\} \\
\{1',\ov{1}\} & \ov{2} && \end{array}
\quad\text{and}\quad
\begin{array}{cccc} \{1',\ov{1}\} & \{\ov{1},1\} & 1 & \{1,\ov{2},1'',2''\} \\
 \ov{2} & \ov{2} && \end{array}
\]
of shape $\la$ with entries in ${\bf Q}_{ex}$ correspond to the Hecke
factorizations
\[
s_2\di (s_0s_2)\di (s_2s_1s_0) \di (s_0s_2) \di (s_2s_1)
\quad \mathrm{and} \quad
s_2\di s_2\di(s_0s_1s_2)\di (s_2s_1s_0) \di (s_0s_1) \di s_1
\]
of $w_\la$ and map to the monomials $x^2_1z_1^5z_2^2y_1$ and
$x_1z_1^5z_2^3y_1y_2$, respectively.
\end{example}

\begin{example}
For each integer $k\geq 0$, the signed permutation $s_k$ of length one can
be viewed as the $k$-Grassmannian element $w_\la$ associated to $\la=1$.
When $k$ is positive, Theorem \ref{Cskew} therefore gives
\[
\CG_{s_k}(Z;X,Y) = \sum_H \beta^{|H|-1}x_1^{e_1}\cdots x_k^{e_k}y_1^{e_{k+1}}
\cdots y_k^{e_{2k}}\prod_j z_j^{f_j+g_j}
\]
summed over all sequences
$H=(e_1,\ldots,e_{2k},f_1,g_1,f_2,g_2,\ldots)$ with finite support
such that $e_i,f_j,g_j\in \{0,1\}$ for each $i,j$ and $|H|:=e_1+\cdots
+ e_{2k}+\sum_j(f_j+g_j)>0$. Furthermore, we have
\[
\CG_{s_0}(Z;X,Y) = \sum_D \beta^{|D|-1}\prod_j z_j^{f_j+g_j}
\]
summed over all sequences $D=(f_1,g_1,f_2,g_2,\ldots)$ with finite
support such that $f_j,g_j\in \{0,1\}$ for each $j$ and
$|D|:=\sum_j(f_j+g_j)>0$.
\end{example}

\subsection{Tableau formula for type B skew Grothendieck polynomials}
\label{tBt}

Given that the root systems of types $\text{B}_n$ and $\text{C}_n$
share a common Weyl group $W_n$, and the similarity between equations
(\ref{dbleB}) and (\ref{dbleC}), it is easy to modify the definitions
and theorems in this section to obtain set-valued $k$-tableau formulas
for the type B Grothendieck polynomials $\BG_w(Z;X,Y)$ and
$K$-theoretic Stanley functions $F^B_w(Z)$ indexed by skew
elements $w=w_{\la/\mu}$.  Indeed, it suffices to change the
definition of a $\ov{z}$-strip in type B to be the same as that of a
$y$-strip. With this modification, the main results of this section
then apply verbatim to the setting of type B.

\begin{example}
The type B Grothendieck polynomials for the elements $s_k$ of length
one in $S_\infty$ satisfy $\BG_{s_k}(Z;X,Y) =
\CG_{s_k}(Z;X,Y)$. Furthermore, we have
\[
\BG_{s_0}(Z;X,Y) = \sum_{D'} \beta^{|D'|-1}\prod_j z_j^{f_j}
\]
summed over all sequences $D'=(f_1,f_2,\ldots)$ with finite support
such that $f_j\in \{0,1\}$ for each $j$ and $|D'|:=\sum_j f_j>0$.
\end{example}

\section{Tableau formula for type D skew Grothendieck polynomials}
\label{tDt}

In this section, we assume that all Grassmannian
and skew elements are taken with respect to a fixed positive integer $k$.

Let $\la$ and $\mu$ be any two typed $k$-strict partitions such that
$(\la,\mu)$ is a compatible pair and let $w_{\la/\mu}$ be the
corresponding skew element of $\wt{W}_n$.  It was shown in \cite{T3}
that $w_{\la/\mu}$ is unimodal if and only if $\la/\mu$ is a typed
$k'$-horizontal strip, and in this case we have
$n'(w_{\la/\mu})=n'(\la/\mu)$.  If $w_{\la/\mu}$ is decreasing down to
$1$, then we say that the typed $k'$-horizontal strip $\la/\mu$ is an
{\em typed $x$-strip} or a {\em typed $z$-strip}. If $w_{\la/\mu}$ is
increasing up from $1$ (resp. $\Box$), then we say that $\la/\mu$ is a
{\em typed $y$-strip} (resp.\ {\em typed $\ov{z}$-strip}).  We say
that a typed $k'$-horizontal strip $\la/\mu$ is {\em extremal} if
\[
(\ell_k(\la),\type(\la))\neq (\ell_k(\mu),\type(\mu)).
\]  
For any typed $k$-strict partition $\la$, let
$\epsilon(\la):=\ell_k(\la)+\type(\la)$.

According to \cite[Prop.\ 2]{T6}, a typed $k'$-horizontal strip
$\la/\mu$ is a typed $x$-strip (resp.\ typed $y$-strip) if and only if
(i) the left boxes in $\la/\mu$ form a vertical strip
(resp.\ horizontal strip), and no two boxes in $\la/\mu$ are
$(k-1)$-related (resp.\ no two right boxes in $\la/\mu$ are in the
same row), and (ii) if $\la/\mu$ is extremal then
$(\type(\la),\type(\mu))\neq (0,0)$ and the following condition holds:
if $\epsilon(\mu)$ is odd, then $\epsilon(\la)$ is odd {\em and}
$\type(\mu)=0$, while if $\epsilon(\mu)$ is even, then $\epsilon(\la)$
is odd {\em or} $\type(\mu)=1$. Moreover, it was shown in the proof of
loc.\ cit.\ that the typed $\ov{z}$-strips are characterized among all
typed $k'$-horizontal strips by the next Proposition.

\begin{prop}
\label{zdistinct}
A typed $k'$-horizontal strip $\la/\mu$ is a typed $\ov{z}$-strip if
and only if {\em (i)} the left boxes in $\la/\mu$ form a horizontal
strip, and no two right boxes in $\la/\mu$ are in the same row, and
{\em (ii)} if $\la/\mu$ is extremal then $(\type(\la),\type(\mu))\neq
(0,0)$ and the following condition holds: if $\epsilon(\mu)$ is odd,
then $\epsilon(\la)$ is even {\em or} $\type(\mu)=1$, while if
$\epsilon(\mu)$ is even, then $\epsilon(\la)$ is even {\em and}
$\type(\mu)=0$.
\end{prop}

The following results are proved exactly as their analogues in type C.

\begin{prop}
\label{skewfactorD}
Suppose that $w_{\la/\mu}=u\di v$ for some skew element $w_{\la/\mu}$
of $\wt{W}_\infty$.  Then there exist typed $k$-strict partitions
$\rho$ and $\nu$ such that $\mu\subset\rho\subset\la$ and
$\mu\subset\nu\subset\la$ are compatible triples, $u=w_{\la/\rho}$, and
$v=w_{\nu/\mu}$.
\end{prop}

\begin{cor}
\label{Dfact}
For the skew element $w_{\la/\mu}\in \wt{W}_\infty$, we have
\[
\DG_{w_{\la/\mu}}(Z;X,Y) = \sum_{\rho,\nu} \DG_{w_{\nu/\mu}}(Z;X)\AG_{w_{\la/\rho}^{-1}}(Y)
\]
summed over all typed $k$-strict partitions $\rho$ and $\nu$ such that
$\mu\subset\rho\subset\la$ and $\mu\subset\nu\subset\la$ are
compatible triples, $w_{\la/\rho}\in S_{\infty}$, and
$w_{\la/\mu}=w_{\la/\rho}\di w_{\nu/\mu}$.
\end{cor}

Suppose that $\la$ is a typed $k$-strict partition and $i\in \N_\Box$
is such that $\ell(s_iw_\la)=\ell(w_\la)-1$. Then $s_iw_\la=w_\mu$ for
a typed $k$-strict partition $\mu\subset\la$, so that $\la/\mu$ is a
typed $k'$-horizontal strip consisting of a single box $b$ of
$\la$. We then say that $b$ is a {\em removable box} of $\la$. It
follows that a box $b$ of $\la$ is a removable box of $\la$ if and
only if (i) $\la\ssm b$ is a partition and (ii) if $b$ is a left box
of $\la$, then $b$ is not $(k-1)$-related to two right boxes of $\la$
(including boxes in row zero) which are both bottom boxes in their
column. Notice that the type of the typed $k$-strict partition
$\mu=\la\ssm b$ is determined by the condition
$\type(\la)+\type(\mu)\neq 3$.

Similarly, suppose that $(\la,\mu)$ is a compatible pair of typed
$k$-strict partitions and $i\in \N_\Box$ is such that
$\ell(s_iw_{\la/\mu})=\ell(w_{\la/\mu})-1$. Then
$s_iw_{\la/\mu}=w_{\nu/\mu}$ for a typed $k$-strict partition $\nu$
with $\mu\subset\nu\subset\la$, so that $\la/\nu$ is a typed
$k$-horizontal strip consisting of a single box $b$ of $\la$, which is
not in $\mu$. We then say that $b$ is a {\em removable box} of
$\lambda/\mu$, and we must have $\type(\la)+\type(\nu)\neq 3$. In this
case, $b$ is a removable box of $\la$ which is not in $\mu$ and such that
$(\la\ssm b,\mu)$ is a compatible pair.

\begin{prop}
\label{skewfactD}
Suppose that $u$ is decreasing down to $1$, increasing up from $1$, or
increasing up from $\Box$, and $w_{\la/\mu}=u\di v$ for some skew
element $w_{\la/\mu}$. Then there exist typed $k$-strict partitions
$\rho\subset\nu$ such that $\mu\subset\rho\subset\la$ and
$\mu\subset\nu\subset\la$ are compatible triples, $\nu/\rho$ is a
$k$-rook strip, $u=w_{\la/\rho}$, and $v=w_{\nu/\mu}$.
\end{prop}
\begin{proof}
The argument is the same as in the proof of Proposition \ref{skewfactC}.
\end{proof}

\begin{remark}
In the situation of Proposition \ref{skewfactD}, if $u$ is decreasing
down to $1$ (resp.\ increasing up from $1$ or increasing up from
$\Box$), then $\la/\nu$ and $\la/\rho$ are both typed $z$-strips
(resp.\ typed $y$-strips or typed $\ov{z}$-strips).
\end{remark}

As in the previous section, Proposition \ref{skewfactD} admits a
converse statement.

\begin{prop}
\label{skewDconverse}
Suppose that $\mu\subset\rho\subset\nu\subset\la$ are typed $k$-strict
partitions such that $\mu\subset\rho\subset\la$ and
$\mu\subset\nu\subset\la$ are compatible triples, $\la/\rho$ is a
typed $z$-strip (resp.\ typed $y$-strip or
typed $\ov{z}$-strip), and $\nu/\rho$ is a $k$-rook strip. Then we have a
Hecke factorization $w_{\la/\mu}=w_{\la/\rho}\di w_{\nu/\mu}$.
\end{prop}
\begin{proof}
The argument is the same as in the proof of Proposition \ref{skewCconverse}.
\end{proof}

Let $\bR$ denote the ordered alphabet 
\[
\ov{1}, \ov{1}^\circ < 1,1^\circ < \ov{2}, \ov{2}^\circ < 2,2^\circ
< \cdots . 
\]
The symbols $\ov{1},\ov{1}^\circ,\ov{2},\ov{2}^\circ,\ldots$ are {\em
  barred}, while the symbols $1,1^\circ, 2,2^\circ,\ldots$ are {\em
  unbarred}.

\begin{defn}
\label{svktabD}
A {\em set-valued typed $k'$-tableau} $T$ of shape $\la/\mu$ is a pair
of sequences of typed $k$-strict partitions
\[
\mu = \la^0 \subset \la^1 \subset \cdots \subset \la^{2r} =\la
\quad \mathrm{and} \quad
\mu = \nu^0, \nu^1, \ldots, \nu^{2r} = \la 
\]
with the first one compatible, such that (i)
$\mu\subset\nu^i\subset\la^i$, $(\nu^i,\mu)$ is a compatible pair, and
$\la^i/\nu^i$ is a $k$-rook strip for each $i$, and (ii)
$\la^i/\nu^{i-1}$ is a typed $\ov{z}$-strip if $i$ is odd and a typed
$z$-strip if $i$ is even, for $1\leq i\leq 2r$.  We represent $T$ by a
filling of the boxes in $\la/\mu$ with finite nonempty subsets of
symbols of $\bR$ such that for each $h\in [1,r]$, the boxes in $T$
which contain the entry $\ov{h}$ or $\ov{h}^\circ$ (resp.\ $h$ or
$h^\circ$) form the skew diagram $\la^{2h-1}/\nu^{2h-2}$
(resp.\ $\la^{2h}/\nu^{2h-1}$), and we use $\ov{h}^\circ$
(resp.\ $h^\circ$) if and only if $\type(\la^{2h-1})=2$
(resp.\ $\type(\la^{2h})=2$), for each $h\in [1,r]$.  For any
set-valued typed $k'$-tableau $T$ we let $|T|$ be the total number of
symbols of $\bR$ which appear in $T$, including their multiplicities,
and set $z^T:=\prod_h z_h^{m_h}$, where $m_h$ denotes the number of
times that $\ov{h}$, $h$, $\ov{h}^\circ$, or $h^\circ$ appears in $T$.
\end{defn}

\begin{thm}
\label{FDthm}
Let $w=w_{\la/\mu}$ be the skew element associated to the compatible pair $(\la,\mu)$ of
typed $k$-strict partitions. Then we have
\begin{equation}
\label{FtabeqD}
F^D_w(Z) = \sum_T \beta^{|T|-|\la/\mu|}\, z^T
\end{equation}
where the sum is over all set-valued typed $k'$-tableau of shape $\la/\mu$.
\end{thm}
\begin{proof}
It is easy to see from the definition of $F^D_w$ that
\begin{equation}
\label{FtabeqD2}
F^D_w(Z) = \sum_{w=\s_{2r} \di\cdots\di \s_1}
\be^{\sum_i \ell(\s_i)-\ell(w)}\prod_j z_j^{\ell(\s_{2j-1})+\ell(\s_{2j})}
\end{equation}
where the inner sum is over all Hecke factorizations $\s_{2r}\di \cdots
\di \s_1$ of $w$, for varying $r$, such that $\s_{2j-1}$ is increasing
up from $\Box$ and $\s_{2j}$ is decreasing down to $1$ for all $j\in
[1,r]$. Propositions \ref{skewfactD} and \ref{skewDconverse} imply
that any such factorization $w=\s_{2r}\di \cdots \di\s_1$ corresponds to a
pair of sequences of typed $k$-strict partitions
\[
\mu = \la^0 \subset \la^1 \subset \cdots \subset \la^{2r} =\la
\quad \mathrm{and} \quad
\mu = \nu^0, \nu^1, \cdots, \nu^{2r} = \la 
\]
as in Definition \ref{svktabD}, with $\s_i= w_{\la^i/\nu^{i-1}}$ for each $i\in [1,2r]$.
Furthermore, if $T$ is the associated set-valued typed $k'$-tableau, then
\[
|T|= \sum_{i=1}^{2r}|\la^i/\nu^{i-1}| = \sum_{i=1}^{2r}\ell(\s_i)
\quad \text{and} \quad
z^T=\prod_{j=1}^r z_j^{\ell(\s_{2j-1})+\ell(\s_{2j})}
\]
while $\ell(w)=|\la/\mu|$. Therefore, equation (\ref{FtabeqD}) follows directly
from (\ref{FtabeqD2}).
\end{proof}

\begin{example}
\label{exDtab}
Given a compatible pair $(\la,\mu)$ of typed $k$-strict partitions, a
{\em barred typed $k'$-tableau}
$T$ of shape $\la/\mu$ is a sequence of typed $k$-strict partitions
\[
\mu = \la^0 \subset \la^1 \subset \cdots \subset \la^{2r} =\la
\]
such that $\la^i/\la^{i-1}$ is a typed $\ov{z}$-strip if $i$
is odd and a typed $z$-strip if $i$ is even, for $1\leq
i\leq 2r$.  We represent $T$ by a filling of the boxes in $\la/\mu$
with symbols in $\bR$ such that for each $h\in [1,r]$, the
boxes in $T$ with entry $\ov{h}$ or $\ov{h}^\circ$ (resp.\ $h$ or
$h^\circ$) form the skew diagram $\la^{2h-1}/\la^{2h-2}$
(resp.\ $\la^{2h}/\la^{2h-1}$), and we use $\ov{h}^\circ$
(resp.\ $h^\circ$) if and only if $\type(\la^{2h-1})=2$
(resp.\ $\type(\la^{2h})=2$), for each $h\in [1,r]$.

If we set $\beta:=0$, then the $K$-theoretic Stanley function
$F^D_w(Z)$ specializes to the type D Stanley function $E_w(Z)$,
introduced and studied in \cite{BH, La}. If $w:=w_{\la/\mu}$ then
Theorem \ref{FDthm} gives
\begin{equation}
\label{FeqD}
E_w(Z) = \sum_T z^T
\end{equation}
summed over all barred typed $k'$-tableaux $T$ of shape $\la/\mu$. The barred
typed $k'$-tableaux and equation (\ref{FeqD}) refine the typed $k'$-tableaux of
\cite[Def.\ 3]{T3} and \cite[Eqn.\ (11)]{T6} in a manner directly analogous to
Example \ref{exCtab}.

For instance, let $k:=1$, $\la:=(3,1)$ with $\type(\la):=1$ so that
$w_\la=2\,\ov{3}\,\ov{1}$, and $\bR_{12}$ denote the alphabet
$\{\ov{1},\ov{1}^\circ<1,1^\circ<\ov{2},\ov{2}^\circ<2,2^\circ\}$.
There are four barred typed $1'$-tableaux of shape $\la$ with entries
in $\bR_{12}$:
\begin{gather*}
\begin{array}{l} \ov{1}^\circ\,1\,1 \\ \ov{2} \end{array}  \ \ \ \
\begin{array}{l} \ov{1}^\circ\,1\,\ov{2} \\ \ov{2} \end{array}  \ \ \ \
\begin{array}{l} 1\,\ov{2}\,2 \\ 1 \end{array}  \ \ \ \
\begin{array}{l} 1\,\ov{2}\,2 \\ \ov{2} \end{array}
\end{gather*}
corresponding to the reduced factorizations $s_\Box \cdot s_2 s_1 \cdot s_\Box$,
$s_\Box s_2 \cdot s_1 \cdot s_\Box$, $s_2\cdot s_\Box \cdot s_2s_1 \cdot 1$, and $s_2\cdot
s_\Box s_2\cdot s_1\cdot 1$
of $w_\la$, respectively.
Moreover, let $\la^\circ:=(3,1)$ with
$\type(\la^\circ):=2$ so that $w_{\la^\circ}=\ov{2}\,\ov{3}\,1$. Then we
find the following four barred typed $1'$-tableaux of shape
$\la^\circ$ with entries in $\bR_{12}$:
\begin{gather*}
\begin{array}{l} \ov{1}^\circ\,1\,1 \\ 2^\circ \end{array}  \ \ \ \
\begin{array}{l} \ov{1}^\circ\,1\,\ov{2} \\ 2^\circ \end{array}  \ \ \ \
\begin{array}{l} \ov{1}^\circ\,2^\circ\,2^\circ \\ 1^\circ \end{array}  \ \ \ \
\begin{array}{l} \ov{1}^\circ\,2^\circ\,2^\circ \\ \ov{2}^\circ \end{array}
\end{gather*}
corresponding to the reduced factorizations $s_1 \cdot 1 \cdot s_2
s_1\cdot s_\Box$, $s_1 \cdot s_2 \cdot s_1 \cdot s_\Box$, $s_2s_1
\cdot 1 \cdot s_2\cdot s_\Box$, and $s_2s_1\cdot s_2\cdot 1 \cdot
s_\Box$ of $w_{\la^\circ}$, respectively. Setting $Z_2:=(z_1,z_2)$, we deduce that
\[
E_{w_\la}(Z_2) = E_{w_{\la^\circ}}(Z_2) = z_1^3z_2+2z_1^2z_2^2+z_1z_2^3.
\]
The reader may compare this with \cite[Example 2]{T3}.
\end{example}

Let {\bf R} denote the ordered alphabet 
\[
(n-1)'<\cdots 2'<1'< \ov{1}, \ov{1}^\circ < 1,1^\circ < \ov{2},\ov{2}^\circ <
2,2^\circ < \cdots <1''<2''<\cdots<(n-1)'' 
\]
which extends the alphabet $\bR$.

\begin{defn}
\label{Dtrit}
Consider a pair $U$ of sequences of typed $k$-strict partitions
\[
\mu = \la^0 \subset \la^1 \subset \cdots \subset \la^{2n+2r-2} =\la
\quad \mathrm{and} \quad
\mu = \nu^0, \nu^1, \ldots, \nu^{2n+2r-2} = \la 
\]
with the first one compatible, such that (i)
$\mu\subset\nu^i\subset\la^i$, $(\nu^i,\mu)$ is a compatible pair, and
$\la^i/\nu^i$ is a $k$-rook strip for each $i$, and (ii)
$\la^i/\nu^{i-1}$ is a typed $x$-strip for each $i\leq n-1$
(non-extremal if $i\leq n-2$), a typed $y$-strip for each $i\geq n+2r$
(non-extremal if $i\geq n+2r+1$), and a typed $\ov{z}$-strip
(resp.\ typed $z$-strip) if $i = n-1+j$ for some odd (resp.\ even)
$j\in [1,2r]$. We represent $U$ by a filling of the boxes in $\la/\mu$
by including the entry $(n-i)'$ in each box of $\la^i/\nu^{i-1}$ for
$1\leq i\leq n-1$, the entry $(i-n-2r+1)''$ in each box of
$\la^i/\nu^{i-1}$ for $n+2r\leq i\leq 2n+2r-2$, and the entry $\ov{j}$
or $\ov{j}^\circ$ (resp.\ $j$ or $j^\circ$) in each box of
$\la^{n-1+j}/\nu^{n-2+j}$ for all odd (resp.\ even) integers $j\in
[1,2r]$, so that the circled entries $j^\circ$ and $\ov{j}^\circ$ are
used if and only if the ambient partition has type 2, as in Definition
\ref{svktabD}.  We say that $U$ is a {\em set-valued typed
  $k'$-tritableau} of shape $\la/\mu$ if for $1\leq i \leq
\ell_k(\mu)$ (resp.\ $1\leq i \leq \ell_k(\la)$) and $1\leq j \leq k$,
the entries of $U$ in row $i$ are $\geq (\mu_i-k+1)'$ (resp.\ $\leq
(\la_i-k)''$) and the entries in column $k+1-j$ lie in the interval
$[|w_\mu(j)|', |w_\la(j)-1|'']$.  We let $|U|$ be the total number of
symbols of {\bf R} which appear in $U$, including their
multiplicities, and define
\[
(xyz)^U:= \prod_ix_i^{n'_i}\prod_i y_i^{n''_i}\prod_j z_j^{n_j} 
\]
where $n'_i$ and $n''_i$ denote the number of times that $i'$ and
$i''$ appear in $U$, respectively, for each $i\in [1,n-1]$, and $n_j$
denotes the number of times that $j$, $j^\circ$, $\ov{j}$, or $\ov{j}^\circ$
appears in $U$, for each $j \in \N_\Box$.
\end{defn}

\begin{thm}
\label{Dskew} For the skew element $w:=w_{\la/\mu}$ of $\wt{W}_n$, we have
\begin{equation}
\label{DGteq}
\DG_w(Z;X,Y)=\sum_U \be^{|U|-|\la/\mu|}\,(xyz)^U
\end{equation}
summed over all set-valued typed $k'$-tritableaux $U$ of shape $\la/\mu$.
\end{thm}
\begin{proof}
We deduce from formula (\ref{dbleD}) that $\DG_w(Z;X,Y)$ is equal to 
\[
\sum \be^{\ell(u,v,\sigma,w)} y_{n-1}^{\ell(v_{n-1})}\cdots y_1^{\ell(v_1)}
x_1^{\ell(u_1)}\cdots x_{n-1}^{\ell(u_{n-1})}\prod_j z_j^{\ell(\sigma_{2j-1})+\ell(\s_{2j})}
\]
where the sum is over all Hecke factorizations
\[
v_{n-1}\di\cdots
\di v_1\di \sigma_{2r}\di\cdots\di\sigma_1 \di u_1\di\cdots \di u_{n-1}
\]
of $w$, for varying $r\geq 0$, such that $v_p\in S_n$ is increasing up
from $p$ and $u_p\in S_n$ is decreasing down to $p$ for each $p$,
$\sigma_i$ is increasing up from $\Box$ (resp.\ decreasing down to
$1$) for all odd $i$ (resp.\ all even $i$), and
$\ell(u,v,\sigma,w):=\sum_p\ell(u_p)+\sum_p\ell(v_p)
+\sum_i\ell(\sigma_i)- \ell(w)$. Propositions \ref{skewfactD} and
\ref{skewDconverse} imply that such factorizations correspond to pairs
of sequences of typed $k$-strict partitions
\[
\mu = \la^0 \subset \la^1 \subset \cdots \subset \la^{2n+2r-2} =\la
\quad \mathrm{and} \quad
\mu = \nu^0, \nu^1, \ldots, \nu^{2n+2r-2} = \la
\]
and an associated filling $U$ of the the boxes in $\la/\mu$ as in
Definition \ref{Dtrit}. The required bounds on the entries of $U$ are
established as in the proof of \cite[Thm.\ 3]{T6}, showing that $U$ is
a set-valued typed $k'$-tritableau of shape $\la/\mu$ such that
\[
(xyz)^U=y_{n-1}^{\ell(v_{n-1})}\cdots y_1^{\ell(v_1)}
x_1^{\ell(u_1)}\cdots x_{n-1}^{\ell(u_{n-1})}\prod_j z_j^{\ell(\sigma_{2j-1})+\ell(\s_{2j})}.
\]
Conversely, the typed set-valued $k'$-tritableaux of shape $\la/\mu$
correspond to Hecke factorizations of $w$ of the required form, and
equation (\ref{DGteq}) follows.
\end{proof}

\begin{example}
  Let $k:=1$ and $\la:=(4,2)$ so that $w_\la=1\,\ov{4}\,\ov{2}\,3
  =s_1s_3s_2s_\Box s_2s_1$, and consider the alphabet ${\bf
    R}_{ex}:=\{1'<\ov{1},\ov{1}^\circ<1, 1^\circ<\ov{2},\ov{2}^\circ<
  2, 2^\circ <1''<2''\}$.  The set-valued typed $1'$-tritableau
\[
\begin{array}{cccc} 1' & \{\ov{1},\ov{2}\} & 2 & \{2,1'',2''\} \\
\{1',\ov{1}\} & 1'' && \end{array}
\quad\text{and}\quad
\begin{array}{cccc} 1' & \{\ov{1},1\} & 1 & \{1,\ov{2},2\} \\
 \ov{2} & \{2,1''\} && \end{array}
\]
of shape $\la$ with entries in ${\bf R}_{ex}$ correspond to the Hecke
factorizations
\[
s_3\di (s_1s_3)\di (s_3s_2) \di s_\Box \di 1 \di
(s_\Box s_2) \di (s_2s_1) \quad \mathrm{and} \quad
s_1 \di (s_3s_1)\di(s_\Box s_3)\di (s_3s_2s_1) \di s_\Box \di s_1
\]
of $w_\la$ and map to the monomials
$x^2_1z_1^2z_2^3y_1^2y_2$ and $x_1z_1^4z_2^4y_1$, respectively.
\end{example}

\begin{example}
Using Theorem \ref{Dskew}, the type D Grothendieck polynomials for the
elements $s_r$ of length one in $\wt{W}_\infty$ are computed as
follows. If $r\geq 2$, then we have
\[
\DG_{s_r}(Z;X,Y) = \CG_{s_r}(Z;X,Y) = \BG_{s_r}(Z;X,Y).
\]
For $r=1$, we have 
\[
\DG_{s_1}(Z;X,Y) = \sum_{H'} \beta^{|H'|-1}x_1^{e_1}y_1^{e_2}
\prod_j z_j^{f_j}
\]
summed over all sequences $H'=(e_1,e_2,f_1,f_2,f_3,\ldots)$ with
finite support such that $e_i,f_j\in \{0,1\}$ for each $i,j$ and
$|H'|:=e_1+ e_2+\sum_j f_j>0$. Finally, for $r=\Box$ we have
\[
\DG_{s_\Box}(Z;X,Y) = \sum_{D'} \beta^{|D'|-1}\prod_j z_j^{f_j}
\]
summed over all sequences $D'=(f_1,f_2,\ldots)$ with finite support such that
$f_j\in \{0,1\}$ for each $j$ and $|D'|:=\sum_j f_j>0$.
\end{example}

\begin{example}
Let $w\in W_\infty$ (respectively $w\in \wt{W}_\infty$) and define the
{\em $K$-theoretic double mixed Stanley functions} $J^B_w$, $J^C_w$,
and (respectively) $I^D_w$ by the equations
\begin{align*}
J^B_w(Z;X,Y) &:=\left\langle A'(Y)B(Z)A(X), w\right\rangle, \\
J^C_w(Z;X,Y) &:=\left\langle A'(Y)C(Z)A(X), w\right\rangle, \ \ \mathrm{and} \\
I^D_w(Z;X,Y) &:=\left\langle A'(Y)D(Z)A(X), w\right\rangle.
\end{align*}
These functions are analogues in connective $K$-theory of the double
mixed Stanley functions $J_w(Z;X,Y)$ and $I_w(Z;X,Y)$ studied in
\cite[Example 3]{T2} and \cite[Sections 3.5 and 4.6]{T4}. Now suppose
that $w:=w_{\la/\mu}$ is a skew element of $W_\infty$ or
$\wt{W}_\infty$. Extend the alphabets {\bf Q} and {\bf R} to include
all primed and double primed positive integers, and modify Definition
\ref{Ctrit} and Definition \ref{Dtrit} by omitting the bounds on the
entries of the tritableaux found in both, and the non-extremal
condition in the latter. Then the right hand sides of equations
(\ref{CGteq}) and (\ref{DGteq}) give tableau formulas for the type C
and type D functions $J^C_w(Z;X,Y)$ and $I^D_w(Z;X,Y)$, respectively.
\end{example}


\begin{thebibliography}{HIMN1}


  
\bibitem[ACT]{ACT} D. Anderson, L. Chen, and N. Tarasca :
{\em $K$-classes of Brill-Noether loci and a determinantal formula},
Int. Math. Res. Not. IMRN 2022, no. 16, 12653--12698.


\bibitem[BH]{BH} S. Billey and M. Haiman :
{\em Schubert polynomials for the classical groups},
J. Amer. Math. Soc. {\bf 8} (1995), 443--482.


\bibitem[BJS]{BJS} S. Billey, W. Jockusch and R. P. Stanley :
{\em Some combinatorial properties of Schubert polynomials},
J. Algebraic Combin. {\bf 2} (1993), 345--374.


\bibitem[B]{B}
A.~S. Buch : {\em A Littlewood-Richardson rule for the $K$-theory of
Grassmannians}, Acta Math. {\bf 189} (2002), 37--78.  


\bibitem[BKTY]{BKTY}
A.~S. Buch, A.~Kresch, H.~Tamvakis, and A.~Yong : {\em Grothendieck
polynomials and quiver formulas}, Amer. J. Math. {\bf 127} (2005), 551--567.


\bibitem[FK1]{FK1}
S.~Fomin and A.~N. Kirillov : {\em Grothendieck polynomials and the
Yang-Baxter equation}, Formal power series and algebraic combinatorics,
183--189, DIMACS, Piscataway, NJ, 1994.


\bibitem[FK2]{FK2}
S.~Fomin and A.~N. Kirillov : \emph{The Yang-Baxter equation,
symmetric functions, and Schubert polynomials}, Proceedings of the
5th Conference on Formal Power Series and Algebraic Combinatorics
(Florence, 1993), Discrete Math. {\bf 153} (1996), 123--143. 


\bibitem[FK3]{FK3} S. Fomin and A. N. Kirillov :
{\em Combinatorial $B_n$-analogs of Schubert polynomials},
Trans. Amer. Math. Soc. {\bf 348} (1996), 3591--3620.

\bibitem[GK]{GK} W. Graham and V. Kreiman :
{\em Excited Young diagrams, equivariant $K$-theory, and Schubert 
varieties}, Trans. Amer. Math. Soc. {\bf 367} (2015), 6597--6645.


\bibitem[HIMN1]{HIMN1} T. Hudson, T. Ikeda, T. Matsumura, and H. Naruse :
{\em Degeneracy loci classes in K-theory--determinantal and Pfaffian formula},
Adv. Math. {\bf 320} (2017), 115--156. 
  
\bibitem[HIMN2]{HIMN2} T. Hudson, T. Ikeda, T. Matsumura, and H. Naruse :
{\em Double Grothendieck polynomials for symplectic and odd orthogonal Grassmannians},
J. Algebra {\bf 546} (2020), 294--314.


\bibitem[IMN]{IMN} T. Ikeda, L. C. Mihalcea, and H. Naruse :
{\em Double Schubert polynomials for the classical groups}, 
Adv. Math. {\bf 226} (2011), 840--886.


\bibitem[IN]{IN} T. Ikeda and H. Naruse :
{\em K-theoretic analogues of factorial Schur $P$- and $Q$-functions},
Adv. Math. {\bf 243} (2013), 22--66. 


\bibitem[KN]{KN} A. N. Kirillov and H. Naruse :
{\em Construction of double Grothendieck polynomials of classical types using
idCoxeter algebras},  Tokyo J. Math. {\bf 39} (2017), 695--728. 


\bibitem[La]{La} T. K. Lam : {\em B and D analogues of stable Schubert
polynomials and related insertion algorithms}, Ph.D.\ thesis, M.I.T., 1995; 
available at http://hdl.handle.net/1721.1/36537.

\bibitem[LP]{LP} T. Lam and P. Pylyavskyy :
{\em Combinatorial Hopf algebras and $K$-homology of Grassmannians},
Int. Math. Res. Not. IMRN 2007, no. 24, Art. ID rnm125, 48 pp.  

\bibitem[LS]{LS} A. Lascoux and M.-P. Sch\"{u}tzenberger :
{\em Structure de Hopf de l'anneau de cohomologie et de l'anneau de
Grothendieck d'une vari\'et\'e de drapeaux}, C. R. Acad. Sci. Paris S\'er.
I Math. {\bf 295} (1982), 629--633. 

\bibitem[Mac]{Mac} I. Macdonald :
{\em Symmetric Functions and Hall Polynomials}, Second edition,
Clarendon Press, Oxford, 1995.

\bibitem[M1]{M1} T. Matsumura : {\em Flagged Grothendieck polynomials},
 J. Algebraic Combin. {\bf 49} (2019), 209--228. 

\bibitem[M2]{M2} T. Matsumura : {\em A tableau formula of double
Grothendieck polynomials for 321-avoiding permutations},
Ann. Comb. {\bf 24} (2020), 55--67.  

\bibitem[Sa]{Sa} B. Sagan :
{\em Shifted tableaux, Schur $Q$-functions, and a conjecture of R. Stanley},
J. Combin. Theory Ser. A {\bf 45} (1987), 62-–103. 

\bibitem[St]{St} J. R. Stembridge :
{\em On the fully commutative elements of Coxeter groups},
J. Algebraic Combin. {\bf 5} (1996), 353--385. 

\bibitem[T1]{T1} H. Tamvakis : 
{\em Giambelli, Pieri, and tableau formulas via raising operators}, 
J. reine angew. Math. {\bf 652} (2011), 207--244.

\bibitem[T2]{T2} H. Tamvakis :
{\em A Giambelli formula for classical $G/P$ spaces}, 
J. Algebraic Geom. {\bf 23} (2014), 245--278.

\bibitem[T3]{T3} H. Tamvakis : 
{\em A tableau formula for eta polynomials},
Math. Annalen {\bf 358} (2014), 1005--1029.

\bibitem[T4]{T4} H. Tamvakis :
{\em Schubert polynomials and degeneracy locus formulas}, 
Schubert Varieties, Equivariant Cohomology and Characteristic Classes,
261--314, EMS Ser. Congr. Rep., Eur. Math. Soc., Z\"urich, 2018. 

\bibitem[T5]{T5} H. Tamvakis : 
{\em Degeneracy locus formulas for amenable Weyl group elements},
arXiv:1909.06398.

\bibitem[T6]{T6} H. Tamvakis : 
{\em Tableau formulas for skew Schubert polynomials}, Bull. Lond. Math. Soc.
{\bf 55} (2023), 1926--1943.

\bibitem[W]{W} D. R. Worley : 
{\em A theory of shifted Young tableaux}, Ph.D. thesis, MIT, 1984;
available at http://hdl.handle.net/1721.1/15599.
  
\bibitem[Y]{Y} D. Yeliussizov : {\em Symmetric Grothendieck
  polynomials, skew Cauchy identities, and dual filtered Young
  graphs}, J. Combin. Theory Ser. A {\bf 161} (2019), 453--485.


\end{thebibliography}
\end{document}